\documentclass[11pt]{article}

\usepackage{amsfonts}
\usepackage{amssymb}
\newcommand{\field}[1]{\mathbb{#1}}

\newcommand{\R}{\field{R}}
\newcommand{\N}{\field{N}}

\newcommand{\be}{\begin{equation}}
\newcommand{\ve}{\varepsilon}
\newcommand{\vp}{\varphi}
\newcommand{\al}{\alpha}
\newcommand{\la}{\lambda}
\newcommand{\om}{\omega}

\newcommand{\A}{{\cal A}}
\newcommand{\B}{{\cal B}}
\newcommand{\F}{{\cal F}}
\newcommand{\G}{{\cal G}}
\newcommand{\J}{{\cal J}}

\newcommand{\de}{\partial}

\newcommand{\im}{\mathop{\rm im}}

\newcommand{\codim}{\mathop{\rm codim}}

\newcommand{\ee}{\end{equation}}
\newcommand{\qed}{\hfill \rule{2.3mm}{2.3mm}}
\newtheorem{theorem}{Theorem}[section]
\newtheorem{lemma}[theorem]{Lemma}

\newtheorem{remark}[theorem]{Remark}

\newcommand{\reff}[1]{(\ref{#1})}

 \newenvironment{proof}{\par\smallbreak{\sl\bf Proof.~}}
 {\unskip\nobreak\hfill \qed \par\medbreak}

\title{
	Forced Frequency Locking for Semilinear Dissipative Hyperbolic PDEs
} 

\newcounter{thesame}
\author{
	Irina Kmit
	\thanks{Institute of Mathematics, Humboldt University of Berlin. On leave from the
		Institute for Applied Problems of Mechanics and Mathematics,
		Ukrainian National Academy of Sciences. 
		Tel.: 
		+49(030)2093-45380.
		{\small   E-mail:
			{\tt kmit@mathematik.hu-berlin.de}
	}}
	\ \ \ Lutz Recke \thanks{Institute of Mathematics, Humboldt University of Berlin.
		{\small   E-mail:
			{\tt recke@mathematik.hu-berlin.de}}
}}

\date{}

\begin{document}
	
	\maketitle

\begin{abstract}
\noindent
This paper concerns the behavior of time-periodic solutions to 1D dissipative autonomous semilinear hyperbolic PDEs under the influence of small time-periodic forcing. We show that the phenomenon of forced frequency locking happens similarly to the analogous phenomena  known for ODEs or parabolic PDEs. However, the proofs are essentially more difficult than for ODEs or parabolic PDEs. In particular,  non-resonance conditions are needed, which do not have counterparts in the cases of ODEs or parabolic PDEs. We derive a scalar equation which answers the main question of forced frequency locking: Which time shifts of the solution to the unforced equation do survive under which forcing?

\end{abstract}

\section{Introduction}
\label{Introduction}
The following phenomenon is usually called forced frequency locking (but also injection locking, 
master-slave synchronization, or frequency entrainment): 
If $u^0$ is a $T^0$-periodic solution to a
nonlinear autonomous evolution equation, if this solution is locally unique up to phase shifts (for example if $\{u^0(t): t \in \R\}$ is a limit cycle), and 
if the autonomous equation is forced by a periodic forcing with intensity $\ve \approx 0$ and period $T \approx T^0$, then generically there exist $\tau \in \R$ (the so-called scaled period deviation) and $\vp \in \R$ (the so-called asymptotic phase) such that for all $\ve$ and $T$ with
\be
\label{coneprop}
\ve \approx 0 \;\mbox{ and }\; \frac{T-T^0}{\ve} \approx \tau
\ee
there exist $T$-periodic solutions $u=u_{\ve,T}$ of the type
\be
\label{as}
u_{\ve,T}(t)\approx u^0(tT^0/T+\vp)
\ee
to the forced equation.
These so-called locked periodic solutions to the forced equation are locally unique and depend continuously on $\ve$ and $T$. Moreover, if $u^0$ is an exponentially orbitally stable periodic solution to the unforced equation, then at least one of the 
locked periodic solutions to the forced equation is exponentially stable.

Forced frequency locking appears in many areas of nature, and it is used in diverse applications in technology, starting with the poineering work of B. van der Pol \cite{Pol}, A. A. Andronov and A. A. Witt \cite{Witt} and R. Adler \cite{Adler}.
Moreover, since a long time it is mathematically rigorously 
described for ODEs with smooth forcing  (see, e.g. \cite{HT78,L50,L59,V80,V83} and \cite[Chapter 5.3.5]{Chi99})
or with Dirac-function like forcing (cf. \cite{Re18}), for functional-differential equations (cf., e.g. \cite{P}) as well as for parabolic PDEs (cf., e.g. \cite{T87}).

It turns out that forced frequency locking appears also in dissipative hyperbolic PDEs, but up to now there does not exist a rigorous mathematical description. Reasons for that seem to be the following.

First, the question, if a nondegenerate time-periodic solution to a dissipative nonlinear hyperbolic PDE is locally unique (up to time shifts in the autonomous case) and depends continuously or even smoothly on the system
parameters, is much more delicate than for ODEs or parabolic PDEs (cf., e.g. \cite{Raugel,Raugel1,KR4}). In particular, for smoothness of the data-to-solution map of hyperbolic PDEs it is necessary, in general,
that the equation depends smoothly not only on the unknown function $u$, but also on the space variable $x$
(and the time variable $t$ in the non-autonomous case).
This is completely different to what is known for parabolic PDEs (cf. \cite{GR}).  Similarly, the description of Hopf bifurcation for dissipative hyperbolic PDEs is much more complicated than for ODEs or parabolic PDEs (cf. \cite{KR3,KR3a,Kos1,Kos2}). 

Second, linear autonomous hyperbolic partial differential operators with one space dimension essentially differ from those with more than one space dimension:
They satisfy the spectral mapping property (see \cite{Lopes} in $L^p$-spaces and, more important
for applications to nonlinear problems, \cite{Lichtner1} in $C$-spaces)
and they generate Riesz bases (see, e.g. \cite{Guo,Mogul}), what is not the case,
in general, if the space dimension is larger than one (see the celebrated counter-example of M. Renardy in \cite{Renardy}).
Therefore, the question of  Fredholmness of those operators in appropriate spaces of time-periodic functions is highly difficult.

The main consequence (from the point of view of applied mathematical techniques) of the fact, that the space dimension of  the hyperbolic PDEs, which we consider, is one, consists in the following: We can use integration along characteristics
in order to replace the nonlinear PDEs by nonlinear partial integral equations (see, e.g. \cite{Appell1} for the notion ``partial integral equation'').
After that, we can apply
known Fredholmness criteria to the linearized partial integral equations (see \cite[Corollary 4.11]{KR4}
and \cite[Theorem 1.2]{second}). 
Here we need the strict hyperbolicity condition (for first-order systems it is \reff{aassnot}, and for second-order equations it is automatically fulfilled), which does not have a counterpart 
in the case of ODEs or parabolic PDEs.

And third, we have to assume conditions (for first-order systems 
it is \reff{nonressys} or \reff{nonressys1}, 
and for second-order equations it is \reff{nonreseq} or
\reff{nonreseq1})
  implying that small divisors, related to the linear parts of the unforced equation 
  (linearization in the periodic solution of the unforced equation),
cannot come too close to zero. There are no counterparts to that condition
in the case of ODEs or parabolic PDEs.\\

In the present paper we will describe forced frequency locking for boundary value problems for 1D semilinear first-order hyperbolic systems of the type
\be
\label{sysold}
\left.
\begin{array}{l}
\partial_tu_j(t,x)+a_j(x)\partial_xu_j(t,x)+b_j(x,u(t,x))=\ve f_j(t/T,x), \;j=1,2, \;x \in [0,1],\\
u_1(t,0)=r_1u_2(t,0)+\ve g_1(t/T),\;
u_2(t,1)=r_2u_1(t,1)+\ve g_2(t/T)
\end{array}
\right\}
\ee
as well as for 1D semilinear second-order hyperbolic equations of the type
\be
\label{eqold}
\left.
\begin{array}{l}
\partial_t^2u-a(x)^2\partial_x^2u+b(x,u,\partial_tu,\partial_xu)=\ve f(t/T,x), \;x \in [0,1],\\
u(t,0)=\ve g_1(t/T),\;
\partial_xu(t,1)+\gamma u(t,1)=\ve g_2(t/T).
\end{array}
\right\}
\ee
We suppose that the  unforced problems, i.e.  \reff{sysold} and \reff{eqold} with $\ve=0$, have time-periodic solutions $u^0$ with period one, that the functions $f_j(\cdot,x)$ and $g_j$
in \reff{sysold} and the functions
$f(\cdot,x)$ and $g_j$ in  \reff{eqold}
are 1-periodic, i.e. that  the functions $f_j(\cdot/T,x)$ and $g_j(\cdot/T)$ in \reff{sysold}
and the functions $f(\cdot/T,x)$ and $g_j(\cdot/T)$ in \reff{eqold}
are $T$-periodic, and we show how to find $\tau$ and $\vp$ such that, for all parameters $\ve$ and $T$ with \reff{coneprop}  and with $T^0=1$, there exist $T$-periodic in time solutions of the type \reff{as} to problems \reff{sysold} and \reff{eqold}, that those  solutions are locally unique and depend continuously on $\ve$ and $T$, and we describe their asymptotic behavior for $\ve\to 0$.

\begin{remark}
\label{abstractequi}\rm
Using a more abstract point of view, one can say that forced frequency locking is a symmetry breaking problem. Indeed, if one formulates the evolution equation as an abstract equation in a space of functions with fixed time period, then the unforced equation is equivariant with respect to time shifts, i.e. equivariant with respect to an action of the rotation group SO(2) on the function space, and the small forcing breaks this equivariance. The solution $u^0$ to the unforced problem generates a whole family of solutions, the family of all its time shifts, i.e. its orbit under the group action. The phenomenon of forced frequency locking is the phenomenon, that some of the members of the group orbit may survive under appropriate small symmetry breaking perturbations. Hence, in order to answer the main questions (Which members do survive under which forcing?) one can use the machinery of equivariant bifurcation theory (new coordinates in a tubular neighbourhood of the group orbit, scaling, implicit function theorem), and this is, what we are going to do. But there is one necessary condition for this machinery to work: The linearization in $u^0$ of the equation should be Fredholm of index zero in the function space. The verification of this condition in the case of hyperbolic PDEs is much more complicated than for ODEs or parabolic PDEs.
\end{remark}

\begin{remark}
\label{multi}\rm
We do not believe that our main results (Theorems \ref{sys1}--\ref{eq2} below) are true for
multidimensional hyperbolic problems of the type \reff{sysold} and \reff{eqold}, in general. It turns out that in the
multidimensional  case one cannot expect to have locally unique time-periodic solutions  without imposing additional conditions (like prescribed spacial frequency vectors, cf. \cite{Kos2,Kos3}).  In other words: We do not believe that
the linearized multidimensional problems are Fredholm of index zero in 
appropriate spaces of time-periodic functions, in general.
\end{remark}

\begin{remark}
\label{stab}\rm
We do not know if dynamic stability properties of the 1-periodic solutions to the unforced problems become inherited on some of the locked $T$-periodic solutions to the forced problems  \reff{sysold} and \reff{eqold}. For ODEs and parabolic PDEs this is known to be true, and we expect that this is true also for \reff{sysold} and \reff{eqold}. But for showing this, we should consider the initial-boundary value problems corresponding to \reff{sysold} and \reff{eqold}, what we do not do in the present paper.
\end{remark}

\begin{remark}
\label{discont}\rm
In many technical applications the forcing is discontinuous, for example step-like, in time (and, in the case of PDEs, in space also). In cases of ODEs or parabolic PDEs this does not essentially change the phenomenon of forced frequency locking, only the classical solution setting has to be replaced by an appropriate weak one.
But, unfortunately, we don't know what happens in problems \reff{sysold} and \reff{eqold} if the forcing is discontinuous. It would be important and interesting for applications to know, which discontinuous
forcing in \reff{sysold} and \reff{eqold} destroys the phenomenon of forced frequency locking and which do not destroy.
\end{remark}

\subsection{Results for first-order systems}
\label{Systems}

In order to formulate our results for problem \reff{sysold} we introduce a new scaled time and new scaled unknown functions $u=(u_1,u_2):\R\times [0,1]\to \R^2$ as follows:
$$
t_{new}:=\frac{1}{T}t_{old},\;
u_{new}(t_{new},x):=u_{old}(t_{old},x).
$$
Then the problem of $T$-periodic solutions to \reff{sysold} is transformed into the following one (with $(t,x) \in \R \times [0,1])$):
\be
\label{sys}
\left.
\begin{array}{l}
\displaystyle\frac{1}{T}\partial_tu_j(t,x)+a_j(x)\partial_xu_j(t,x)+b_j(x,u(t,x))=\ve f_j(t,x), \;j=1,2,\\
u_1(t,0)=r_1u_2(t,0)+\ve g_1(t),\;
u_2(t,1)=r_2u_1(t,1)+\ve g_2(t),\\
u(t+1,x)=u(t,x).
\end{array}
\right\}
\ee
Concerning the data of problem \reff{sys} we suppose (for $j=1,2$)
\begin{eqnarray}
\label{aass}
&&a_j \in C^1([0,1]),\; b_j \in C^3([0,1]\times \R),\ r_j \in\R,\;
a_j(x)\not=0 \mbox{ for all } x \in [0,1],\\
\label{aassnot}
&&a_1(x)\not=a_2(x) \mbox{ for all } x \in [0,1]
\end{eqnarray}
and
\be
\label{fass}
\left.
\begin{array}{l}
f_j \in C^1(\R \times [0,1]), 
\;g_j \in C^1(\R),\\
f_j(t+1,x)=f_j(t,x),\; g_j(t+1)=g_j(t)\mbox{ for all } (t,x) \in \R \times [0,1].
\end{array}
\right\}
\ee
Speaking about solutions to \reff{sys} (or its linearizations) we mean classical solutions, i.e. $C^1$-functions $u:\R\times [0,1]\to \R^2$.

Concerning the unforced problem, i.e. \reff{sys} with $\ve=0$ and $T=1$, we suppose that
\be
\label{unull}
\mbox{there exists a solution $u^0$ to \reff{sys} with $\ve=0$
and $T=1$.}
\ee
Further, for $j,k=1,2$ and $t\in \R$ and $x,y \in [0,1]$ we set
\be
\label{bjk}
\al_j(x,y):=\int_x^y\frac{dz}{a_j(z)},\quad
b_{jk}(t,x):=\partial_{u_k}b_j(x,u^0(t,x)),
\ee
and we consider the conditions
\be
\label{nonressys}
|r_1r_2|\not=\exp \int_0^1\left(\frac{b_{11}
(t-\al_1(x,1),x)}{a_1(x)}-
\frac{b_{22}
(t-\al_2(x,1),x)}{a_2(x)}\right)\,dx\quad 
\mbox{for all } t \in \R
\ee
and
\be
\label{nonressys1}
|r_1r_2|\not=\exp \int_0^1\left(\frac{b_{11}
(t+\al_1(0,x),x)}{a_1(x)}-
\frac{b_{22}
(t+\al_2(0,x),x)}{a_2(x)}\right)\,dx\quad 
\mbox{for all } t \in \R.
\ee
Below (cf. Subsection \ref{Regularity}) we will show that, if the assumptions \reff{aass}, \reff{aassnot} and \reff{unull} are satisfied and if  one of the conditions \reff{nonressys} and \reff{nonressys1} is satisfied, then not only the first partial derivatives $\partial_tu^0$ and $\partial_xu^0$ exist and are continuous, but also the second partial derivatives $\partial_t^2u^0$  and $\partial_t\de_xu^0$ exist and are continuous. 
Therefore, the function $u=\partial_tu^0$ solves the linear homogeneous problem
\be
\label{linsys}
\left.
\begin{array}{l}
\partial_tu_j(t,x)+a_j(x)\partial_xu_j(t,x)+
\sum_{k=1}^2b_{jk}(t,x)u_k(t,x)=0, \;j=1,2,\\
u_1(t,0)=r_1u_2(t,0),\;
u_2(t,1)=r_2u_1(t,1),\\
u(t+1,x)=u(t,x).
\end{array}
\right\}
\ee
Hence, the following assumption makes sense:
\be
\label{kersys}
\mbox{For any solution $u$ to \reff{linsys} there exists a constant $c\in \R$ such that $u=c \partial_tu^0$.}
\ee
Remark that, if  assumptions \reff{aass},
\reff{aassnot}
and \reff{unull} are satisfied, but neither \reff{nonressys} nor \reff{nonressys1}, then it may happen that $\partial_t^2u^0$ does not exist (cf. Remark \ref{counter}).
Further, we consider the  linear homogeneous problem adjoint to \reff{linsys}, namely
\be
\label{adsys}
\left.
\begin{array}{l}
-\partial_tu_j(t,x)-\partial_x(a_j(x)u_j(t,x))+
\sum_{k=1}^2b_{kj}(t,x)u_k(t,x)=0, \;j=1,2,\\
r_1a_1(0)u_1(t,0)+a_2(0)u_2(t,0)=0,\;
r_2a_2(1)u_2(t,1)+a_1(1)u_1(t,1)=0,\\
u(t+1,x)=u(t,x),
\end{array}
\right\}
\ee
and we suppose that
\be
\label{keradsys}
\mbox{there exists a solution $u=u^*$ to \reff{adsys} such that } 
\sum_{j=1}^2\int_0^1\int_0^1 \partial_tu_j^0 u_j^* \,dtdx=1.
\ee
It follows from assumption \reff{keradsys} that $\partial_tu^0$ is not the zero function, i.e. that $u^0$ is not constant with respect to time.
 Therefore, assumption \reff{kersys} yields that the space of all solutions to \reff{linsys} has dimension one, and, hence, the space of all solutions to \reff{adsys} has dimension one also. Moreover, assumption \reff{keradsys} means that these two one-dimensional spaces are not $L^2$-orthogonal. In other words: Zero is a geometrically and algebraically simple eigenvalue to the eigenvalue problem corresponding to \reff{linsys}.

In order to formulate our results concerning problem \reff{sys}, we introduce some further notation. We will work with the 1-periodic function $\Phi:\R\to\R$, which is defined by
\begin{eqnarray}
\label{Phidef}
\Phi(\vp)&:=&
-\sum_{j=1}^2\int_0^1\int_0^1f_j(t-\vp,x)u^*_j(t,x)\,dtdx\nonumber\\
&&-\int_0^1\left(a_1(0)
g_1(t-\vp)u_1^*(t,0)
-a_2(1)
g_2(t-\vp)u_2^*(t,1)\right)dt.
\end{eqnarray}
Further, we will work with the maximum norm
$$
\|u\|_\infty:=\max\{|u_j(t,x)|:\; (t,x) \in \R\times [0,1], \; j=1,2\}
$$
for continuous functions $u:\R\times [0,1] \to \R^2$, which are periodic with respect to $t$, and with shift operators $S_\vp$ (for $\vp \in \R$), which work on those functions as
\be 
\label{repdef}
[S_\vp u](t,x):=u(t+\vp,x).
\ee
And finally, for $\ve_0>0$ and $\tau_0 \in \R$ we consider triangles $K(\ve_0,\tau_0) \subset \R^2$, which are defined by
$$
K(\ve_0,\tau_0):=\{(\ve,1+\ve \tau) \in \R^2: \;
\ve \in (0,\ve_0),\; |\tau-\tau_0|<\ve_0\}.
$$

Theorem \ref{sys1} below concerns existence, local uniqueness and continuous dependence of families of solutions $u=u_{\ve,T}$ to \reff{sys} (with $(\ve,T) \in K(\ve_0,\tau_0)$ and with
certain $\ve_0>0$ and $\tau_0 \in \R$), and it describes the asymptotic behavior of $u_{\ve,T}$
for $\ve \to 0$:

\begin{theorem}
\label{sys1}
Suppose \reff{aass}--\reff{unull},
\reff{kersys} and \reff{keradsys}, and assume that one of the conditions \reff{nonressys} and
\reff{nonressys1} is satisfied.
Let $\vp_0,\tau_0 \in \R$ be given such that
\be
\label{beta}
\Phi(\vp_0)=\tau_0 \mbox{ and } \Phi'(\vp_0)\not=0.
\ee
Then the following is true:

(i)  {\bf existence and local uniqueness: }
There exist $\ve_0 >0$ and $\delta>0$ such that for all   $(\ve,T) \in K(\ve_0,\tau_0)$ there exists a unique solution $u=u_{\ve,T}$ 
to \reff{sys} with  
$\|u-S_{\vp_0}u^0\|_\infty<\delta$.

(ii) {\bf continuous dependence: } The map  $(\ve,T) \in K(\ve_0,\tau_0) \mapsto u_{\ve,T}$ 
is continuous with respect to $\|\cdot\|_\infty$.

(iii)  {\bf asymptotic behavior: } We have
$$
\sup_{(\ve,T) \in K(\ve_0,\tau_0)}\frac{1}{\ve}\;\inf_{\vp \in \R}\|u_{\ve,T}-S_\vp u^0\|_\infty<\infty.
$$

(iv) {\bf asymptotic phases: } There exists a continuous function $\tau \in (\tau_0-\ve_0,\tau_0+\ve_0) \mapsto
\vp_\tau \in \R$ with  $\vp_{\tau_0}=\vp_0$ and
\be
\label{asphase}
\Phi(\varphi_\tau)=\tau \mbox{ and }
\lim_{\ve \to 0}\left\|[u_{\ve,T}-S_\vp u^0]_{T=1+\ve \tau,\vp=\vp_\tau}\right\|_\infty=0 \mbox{ for all } \tau \in (\tau_0-\ve_0,\tau_0+\ve_0).
\ee
\end{theorem}

The next theorem claims, roughly speaking, that almost any solution $u$ to \reff{sys} with $\ve \approx 0$, $T\approx 1$ and with $\inf_{\vp\in \R}\|u-S_\vp u^0\|_\infty \approx 0$ is described by one of the solution families from Theorem \ref{sys1}:

\begin{theorem}
\label{sys2}
Suppose \reff{aass}--\reff{unull},
\reff{kersys} and \reff{keradsys}, and  assume that  one of the conditions \reff{nonressys} and
\reff{nonressys1} is satisfied.
Let
$(\ve_k,T_k,u_k)$, $k \in \N$, be a sequence
of solutions to \reff{sys} such that
\be
\label{seqlim}
\lim_{k \to \infty}\left(|\ve_k|+|T_k-1|+\inf_{\vp \in \R}\|u_k-S_\vp u^0\|_\infty\right)=0,
\ee
and let $\vp_k \in [0,1]$, $k \in \N$, be a sequence such that
$\inf_{\vp \in \R}\|u_k-S_\vp u^0\|_\infty=
\|u_k-S_{\vp_k} u^0\|_\infty$.
Then
$$
\lim_{k \to \infty} \left(\Phi(\vp_k)-\frac{T_{k}-1}{\ve_{k}}\right)=0.
$$
In particular, if there exist 
$\vp_0 \in \R$ such that $\vp_k \to \vp_0$ for $k \to \infty$ and  $\Phi'(\vp_0)\not=0$, then for large $k$ we have
$u_k=u_{\ve_k,T_k}$,
where $u_{\ve,T}$ is the family of solutions to \reff{sys}, described by Theorem \ref{sys1}, corresponding to $\vp_0$ and $\tau_0=\Phi(\vp_0)$. 
\end{theorem}

\subsection{Results for second-order equations}
\label{Equations}

If we introduce a new scaled time and new scaled unknown functions $u:\R\times [0,1]\to \R$ as in Subsection \ref{Systems},
then the problem of time-periodic solutions with period $T$ to \reff{eqold} is transformed into the following one (with $(t,x) \in \R \times [0,1])$):
\be
\label{eq}
\left.
\begin{array}{l}
\displaystyle\frac{1}{T^2}\partial_t^2u-a(x)^2\partial_x^2u+b(x,u,\partial_tu/T,\partial_xu)=\ve f(t,x),\\
u(t,0)=\ve g_1(t),\;
\partial_xu(t,1)+\gamma u(t,1)=\ve g_2(t),\\
u(t+1,x)=u(t,x).
\end{array}
\right\}
\ee
Concerning the data of problem \reff{eq} we suppose that
\be
\label{aass1}
\left.
\begin{array}{l}
a \in C^2([0,1]),\; b \in C^3([0,1]\times \R^3),\; \gamma \in \R, \;
a(x)\not=0 \mbox{ for all } x \in [0,1], \\
f \in C^1(\R \times [0,1]), \;g_j \in C^1(\R),\\
f(t+1,x)=f(t,x) \mbox{ and } g_j(t+1)=g_j(t)\mbox{ for all } (t,x) \in \R \times [0,1].
\end{array}
\right\}
\ee
Speaking about solutions to \reff{eq} (or its linearizations), we mean classical solutions again, i.e. $C^2$-functions $u:\R\times [0,1]\to \R$.
We suppose that
\be
\label{unull1}
\mbox{there exists a solution $u^0$ to \reff{eq} with $\ve=0$
and $T=1$.}
\ee
Further, for $t\in \R$ and $x,y \in [0,1]$, we write
\be
\label{bjk1}
\al(x,y):=\int_x^y\frac{dz}{a(z)},\;
b_{j}(t,x):=\partial_jb(x,u^0(t,x),\partial_t u^0(t,x),
\partial_x u^0(t,x)) \mbox{ for } j=2,3,4,
\ee\mbox
where $\partial_jb$ is the partial derivative of the function $b$ with respect to its $j$th variable, i.e.
$\partial_2b$ is the derivative of the function $b$ with respect to the place of $u$,
$\partial_3b$ is the derivative of the function $b$ with respect to the place of $ \partial_tu/T$, and $\partial_4b$ is the derivative of the function $b$ with respect to the place of $\partial_x u$. 
We also write
\be
\label{betadef}
\left.
  \begin{array}{rcl}
\beta^0_1(t,x)&:=&\displaystyle\frac{1}{2}\left(a'(x)+b_3(t,x)+\frac{b_4(t,x)}{a(x)}\right),\;\\
\beta^0_2(t,x)&:=&\displaystyle\frac{1}{2}\left(-a'(x)+b_3(t,x)-\frac{b_4(t,x)}{a(x)}\right),
  \end{array}
  \right\}
  \ee
and we consider the conditions
\be
\label{nonreseq}
 \int_0^1\frac{\beta^0_{1}
(t+\al(x,1),x)+
\beta^0_{2}
(t-\al(x,1),x)}{a(x)}\,dx\not=0\;
\mbox{ for all } t \in \R
\ee
and
\be
\label{nonreseq1}
 \int_0^1\frac{\beta^0_{1}
(t-\al(0,x),x)+
\beta^0_{2}
(t+\al(0,x),x)}{a(x)}\,dx \not=0\;
\mbox{ for all } t \in \R.
\ee
In Subsection \ref{Transfo} we will show that, if one of the conditions \reff{nonreseq} and \reff{nonreseq1} is satisfied, then 
$u^0$ is not only $C^2$-smooth, but even 
the third partial derivatives $\de_t^3u^0$ and
$\de_t\de_x^2u^0$ exist and are continuous. Therefore, the function $u=\partial_tu^0$ solves the linear homogeneous problem
\be 
\label{lineq}
\left.
\begin{array}{l}
\partial^2_tu-a(x)^2\partial^2_xu+
b_{2}(t,x)u+b_{3}(t,x)\partial_tu
+b_{4}(t,x)\partial_xu
=0, \\
u(t,0)=0,\;
\partial_xu(t,1)+\gamma u(t,1)=0,\\
u(t+1,x)=u(t,x).
\end{array}
\right\}
\ee
We assume that
\be
\label{kereq}
\mbox{For any solution $u$ to \reff{lineq} there exists a constant $c\in \R$ such that $u=c \partial_tu^0$.}
\ee
Further, we consider the linear homogeneous problem adjoint to \reff{lineq}, namely
\be
\label{adeq}
\left.
\begin{array}{l}
\partial^2_tu-\partial^2_x(a(x)^2u)+
b_{2}(t,x)u-\partial_t(b_{3}(t,x)u)-\partial_x(b_{4}(t,x)u)
=0,\\
u(t,0)=0,\;
(a(1)(2a'(1)+\gamma a(1))+b_4(t,1))u(t,1)+a(1)^2\partial_xu(t,1)=0,\\
u(t+1,x)=u(t,x).
\end{array}
\right\}
\ee
We suppose that there exists a solution $u=u^*$ to \reff{adeq} such that
\be
\label{keradeq}
\int_0^1\int_0^1
\left(2\partial^2_tu^0(t,x)+b_3(t,x)\de_tu^0(t,x)\right) u^*(t,x)\, dtdx=1.
\ee
Remark that, if  neither \reff{nonreseq} nor \reff{nonreseq1} are satisfied, then it may happen that $\partial_t^3u^0$ does not exist  (cf. Remark \ref{countereq}).

In order to formulate our results concerning problem \reff{eq}, we introduce  further notation. We will work with the 1-periodic function $\Phi:\R\to\R$, which is defined by
\begin{eqnarray}
\label{Phidefeq}
&&\Phi(\vp):=
-\int_0^1\int_0^1f(t-\vp,x)u^*(t,x)\,dtdx
\nonumber\\
&&-\int_0^1\left(a(1)^2g_2(t-\vp)u^*(t,1)+
a(0)^2g_1(t-\vp)\de_xu^*(t,0)\right)dt.
\end{eqnarray}
And again, we will work with the maximum norm
$\|u\|_\infty:=\max\{|u(t,x)|:\; (t,x) \in \R\times [0,1]\}$
for continuous functions $u:\R\times [0,1] \to \R$, which are periodic with respect to $t$, and with shift operators $S_\vp$, which work on those functions as
$[S_\vp u](t,x):=u(t+\vp,x).$

The following two theorems are similar to Theorems \ref{sys1} and \ref{sys2}.

\begin{theorem}
\label{eq1}
Suppose \reff{aass1}, \reff{unull1}, 
\reff{kereq} and \reff{keradeq}, and assume  that  one of  the conditions \reff{nonreseq} and
\reff{nonreseq1} is satisfied.
Let $\vp_0,\tau_0 \in \R$ be given such that
$$
\Phi(\vp_0)=\tau_0 \mbox{ and } \Phi'(\vp_0)\not=0.
$$
Then the following is true:

(i)  {\bf existence and local uniqueness: }
There exist $\ve_0 >0$ and $\delta>0$ such that for all   $(\ve,T) \in K(\ve_0,\tau_0)$ there exists a unique solution $u=u_{\ve,T}$ 
to \reff{eq} with  
$\|u-S_{\vp_0}u^0\|_\infty<\delta$.

(ii) {\bf continuous dependence: } The map  $(\ve,T) \in K(\ve_0,\tau_0) \mapsto u_{\ve,T}$ 
is continuous with respect to $\|\cdot\|_\infty$

(iii)  {\bf asymptotic behavior: } We have
$$
\sup_{(\ve,T) \in K(\ve_0,\tau_0)}\frac{1}{\ve}\;\inf_{\vp \in \R}\|u_{\ve,T}-S_\vp u^0\|_\infty<\infty.
$$

(iv) {\bf asymptotic phases: } There exists a continuous function $\tau \in (\tau_0-\ve_0,\tau_0+\ve_0) \mapsto
\vp_\tau \in \R$ such that  $\vp_{\tau_0}=\vp_0$ and
$$
\Phi(\varphi_\tau)=\tau \mbox{ and }
\lim_{\ve \to 0}\left\|[u_{\ve,T}-S_\vp u^0]_{T=1+\ve \tau,\vp=\vp_\tau}\right\|_\infty=0 \mbox{ for all } \tau \in (\tau_0-\ve_0,\tau_0+\ve_0).
$$
\end{theorem}

\begin{theorem}
\label{eq2}
Suppose \reff{aass1}--\reff{unull1}, 
\reff{kereq} and \reff{keradeq}, and assume  that one of the conditions \reff{nonreseq} and
\reff{nonreseq1} is satisfied.
Let
$(\ve_k,T_k,u_k)$, $k \in \N$, be a sequence
of solutions to \reff{eq} such that
$$
\lim_{k \to \infty}\left(|\ve_k|+|\om_k-1|+\inf_{\vp \in \R}\|u_k-S_\vp u^0\|_\infty\right)=0,
$$
and let $\vp_k \in [0,1]$, $k \in \N$, be a sequence such that
$\inf_{\vp \in \R}\|u_k-S_\vp u^0\|_\infty=
\|u_k-S_{\vp_k} u^0\|_\infty$.
Then
$$
\lim_{k \to \infty} \left(\Phi(\vp_k)-\frac{T_{k}-1}{\ve_{k}}\right)=0.
$$
In particular, if there exist 
$\vp_0 \in \R$ such that $\vp_k \to \vp_0$ for $k \to \infty$, and if $\Phi'(\vp_0)\not=0$, then for large $k$ we have
$u_k=u_{\ve_k,T_k}$,
where $u_{\ve,T}$ is the family of solutions to \reff{eq}, described by Theorem \ref{eq1}, corresponding to $\vp_0$ and $\tau_0=\Phi(\vp_0)$. 
\end{theorem}

\begin{remark}
\label{remad}\rm
For all $u,v \in C_{per}^2(\R\times[0,1];\R)$ with $u$ satisfying the boundary conditions in \reff{lineq}, we have
\begin{eqnarray*}
&&\int_0^1\int_0^1\left(\de_t^2u-a^2\de_x^2u+b_2u+b_3\de_tu+b_4\de_xu\right)v\,dtdx\\
&&-\int_0^1\int_0^1\left(\de_t^2v-\de_x^2(a^2v)+b_2v-\de_t(b_3v)-\de_x(b_4v)\right)u\,dtdx\\
&&=\int_0^1\left[-a^2\de_xuv+(2aa'v+a^2\de_xv)u+b_4uv\right]_{x=0}^{x=1}\,dt\\
&&=\int_0^1\left(\left[((\gamma a^2
+2aa'+b_4)v+a^2\de_xv)u\right]_{x=1}+\left[a^2v\de_xu\right]_{x=0}\right)\,dt.
\end{eqnarray*}
Hence, if $\int_0^1\int_0^1\left(\de_t^2u-a^2\de_x^2u+b_2u+b_3\de_tu+b_4\de_xu\right)
v\,dtdx=0$
for certain $v$ and all $u$, then $v$ satisfies \reff{adeq}. This way one gets the structure of the adjoint boundary value problem \reff{adeq}.
\end{remark}

\subsection{Further remarks}
\label{Remarks}

First we formulate some remarks concerning problem \reff{sys}, but could be stated similarly for problem \reff{eq}.

\begin{remark}
\label{rel}\rm
If $(\ve,T) \in K(\ve_0,\tau_0)$, then $T=1+\ve \tau$ with $\ve \in (0,\ve_0)$ and $\tau \in (\tau_0-\ve_0,\tau_0+\ve_0)$, i.e.
$$
\tau=\frac{T-1}{\ve}
$$
is a scaled period deviation parameter. Hence, the so-called {\bf phase equation} $\Phi(\vp)=\tau$ describes the relationship between the scaled 
period deviation $\tau$ and the corresponding asymptotic phase $\vp=\vp_\tau$
(cf. \reff{asphase}) of the solution family $u_{\ve,T}$. More precisely: The phase equation answers, asymptotically for $\ve \to 0$, the main question of forced frequency locking, i.e. which phase shifts of $u^0$ (described by $\vp$) survive under which forcing (described by $\tau$). Remark that the phase equation, i.e. the function $\Phi$, depends surprisingly explicitly on the shapes of the forcing, i.e.
on the functions $f_j$ and $g_j$. The only data, which are needed for the computation of $\Phi$, but which are not explicitly given, is the solution $u^*$ to the adjoint linearized problem \reff{adsys}. There exists a huge literature about the question, how to compute (numerically or, sometimes, analytically) the solution to the adjoint linearized problem (in the ODE case $u^*$ is a time-periodic vector function, often called perturbation projection vector).
\end{remark}

\begin{remark}
\label{stability}\rm
We expect that the phase equation  $\Phi(\vp)=\tau$ describes not only  the relationship between the scaled period deviation $\tau=(T-1)/\ve$ and the corresponding 
asymptotic phase $\vp$, but also the stability of the locked periodic solutions  $u_{\ve,T}$. More precisely, we expect the following to be true: Let the assumptions of Theorem \ref{sys1} be satisfied, and let $u^0$ be  exponentially orbitally stable as a periodic solution to
$$
\begin{array}{l}
\partial_tu_j(t,x)+a_j(x)\partial_xu_j(t,x)+b_j(x,u(t,x))=0, \;j=1,2,\\
u_1(t,0)=r_1u_2(t,0),\;
u_2(t,1)=r_2u_1(t,1).
\end{array}
$$
Then for all $(\ve,T) \in K(\ve_0,\tau_0)$ with sufficiently small $\ve_0$ the periodic solution $u_{\ve,T}$ to
$$
\begin{array}{l}
\displaystyle \frac{1}{T}\partial_tu_j(t,x)+a_j(x)\partial_xu_j(t,x)+b_j(x,u(t,x))=\ve f_j(t,x), \;j=1,2,\\
u_1(t,0)=r_1u_2(t,0)+\ve g_1(t),\;
u_2(t,1)=r_2u_1(t,1)+\ve g_2(t)
\end{array}
$$
is exponentially stable if
$\Phi'(\vp_0)>0$ and unstable if $\Phi'(\vp_0)<0$. 
For ODEs those results are well-known, cf., e.g., \cite[Theorem 3]{L59} or \cite[Theorem 5.1]{RP96}. 
\end{remark}

\begin{remark}
\label{rel1}\rm
If there exist $\vp_0,\tau_0 \in \R$ with 
\reff{beta}, i.e.
$\Phi(\vp_0)=\tau_0$  and $\Phi'(\vp_0)\not=0$,
then there exist intervals $[\vp_-,\vp_+]$ and 
$[\tau_-,\tau_+]$ such that
\be 
\label{beta1}
\left.
\begin{array}{l}
\vp_-<\vp_0<\vp_+,\; 
\tau_-<\tau_0<\tau_+,\;
\Phi([\vp_-,\vp_+])=[\tau_-,\tau_+],\\
\Phi'(\vp) \not=0
\mbox{ for all } \vp \in [\vp_-,\vp_+],
\end{array}
\right\}
\ee
and Theorem \ref{sys1} works in any $\vp_0
\in [\vp_-,\vp_+]$ with corresponding
$\tau_0=\Phi(\vp_0)$. Moreover, the proof of Theorem \ref{sys1} shows that the constants $\ve_0$ and $\delta$ from assertion $(i)$ of Theorem \ref{sys1} 
may be chosen uniformly with respect to $\vp_0
\in [\vp_-,\vp_+]$. Hence, Theorem \ref{sys1}
can be generalized as follows: 

Suppose \reff{aass}--\reff{unull}, 
\reff{kersys} and \reff{keradsys}, and assume that  one of the conditions \reff{nonressys} and
\reff{nonressys1} is satisfied.
Let
intervals $[\vp_-,\vp_+]$ and 
$[\tau_-,\tau_+]$ be given with \reff{beta1}.
Then the following is true:

$(i)$  There exist $\ve_0,\delta>0$ such that for all   $\ve \in (0,\ve_0)$ and all 
$T>0$  with
$\ve\tau_- \le T-1 \le \ve\tau_+$
there exists a unique solution $u=u_{\ve,T}$ 
to \reff{sys} with  
$\inf \{\|u-S_{\vp}u^0\|_\infty: \,\vp \in [\vp_-,\vp_+]\}<\delta$.

$(ii)$ The map  $(\ve,T)  \mapsto u_{\ve,T}$ 
is continuous with respect to $\|\cdot\|_\infty$

$(iii)$  We have
$$
\sup\left\{\frac{1}{\ve}\;\inf_{\vp \in \R}\|u_{\ve,T}-S_\vp u^0\|_\infty: \; \ve \in (0,\ve_0), \;\tau_- \le \frac{T-1}{\ve} \le\tau_+
\right\}<\infty.
$$

$(iv)$ Let $\vp=\vp_\tau \in [\vp_-,\vp_+]$ be the unique solution to the phase equation $\Phi(\vp)=\tau$
with $\tau \in [\tau_-,\tau_+]$.
Then
\be
\label{straight}
\lim_{\ve \to 0}\left\|[u_{\ve,T}-S_\vp u^0]_{T=1+\ve \tau,\vp=\vp_\tau}\right\|_\infty=0 \mbox{ for all } \tau \in [\tau_-,\tau_+].
\ee
\end{remark}

\begin{remark}
\label{rem}\rm
Assertion \reff{straight} claims that $u_{\ve,1+\ve\tau}$ tends uniformly to 
$S_{\vp_\tau}u^0$ 
for $\ve \to 0$,  
and the limit $S_{\vp_\tau}u^0$ 
depends on $\tau$, in general.
In other words: If the point $(\ve,T)$
tends to the point $(0,1)$ along a straight line, then $u_{\ve,T}$ converges uniformly.
But, in general, $u_{\ve,T}$ does not converge if $(\ve,T)$ tends to $(0,1)$ not along a straight line.
\end{remark}

\begin{remark}
\label{fnonl}\rm
It is easy to show that Theorems \ref{sys1} and \ref{sys2} can be generalized to problems of the type
$$
\begin{array}{l}
\displaystyle\frac{1}{T}\partial_tu_j(t,x)+a_j(x)\partial_xu_j(t,x)+b_j(x,u(t,x))=\ve f_j(t,x,u(t,x)), \;j=1,2,\\
u_1(t,0)=r_1u_2(t,0)+\ve g_1(t,u(t,0)),\;
u_2(t,1)=r_2u_1(t,1)+\ve g_2(t,u(t,1)),\\
u(t+1,x)=u(t,x).
\end{array}
$$
In the definition of the function $\Phi$ one has to replace $g_1(t)$, $g_2(t)$ and $f_j(t,x)$ by
$g_1(t,u^0(t,0))$, $g_2(t,u^0(t,1))$
and $f_j(t,x,u^0(t,x))$, respectively.
\end{remark}

\begin{remark}
\label{quasi}\rm
We believe that Theorems \ref{sys1} and \ref{sys2} can be generalized
to quasilinear systems of the type
$$
\displaystyle\frac{1}{T}\partial_tu_j(t,x)+a_j(x,u(t,x))\partial_xu_j(t,x)+b_j(x,u(t,x))=\ve f_j(t,x,u(t,x)), \;j=1,2,
$$
but probably the proofs will essentially be more difficult than in the semilinear case.
\end{remark}

\begin{remark}
\label{nn}\rm
Theorems \ref{sys1} and \ref{sys2} can be generalized to problems for $n \times n$ first-order hyperbolic systems of the type
(with natural numbers $m<n$)
\be 
\label{nnsys}
\left.
\begin{array}{l}
\displaystyle\frac{1}{T}\partial_tu_j(t,x)+a_j(x)\partial_xu_j(t,x)+b_j(x,u(t,x))=\ve f_j(t,x), \;j=1,\ldots,n,\\
\displaystyle
u_j(t,0)=\sum_{k=m+1}^nr_{jk}u_k(t,0)+\ve g_j(t),\; j=1,\ldots,m,\\
\displaystyle
u_j(t,1)=\sum_{k=1}^mr_{jk}u_k(t,1)+\ve g_j(t),\, j=m+1,\ldots,n,\\
u(t+1,x)=u(t,x),
\end{array}
\right\}
\ee
where for all $x \in [0,1]$ it is supposed $a_j(x)\not=0$ and $a_j(x)\not=a_k(x)$ for $j\not=k$.
Here conditions \reff{nonressys} and
\reff{nonressys1} are replaced by the conditions
$$
\max_{s,t \in [0,1]}
\max_{1 \le j \le m}\sum_{k=m+1}^n\sum_{l=1}^m|r_{jk}r_{kl}|\exp\int_0^1
\left(\frac{\partial_{u_k}b_k(x,u^0(t,x))}{a_k(x)}-
\frac{\partial_{u_j}b_j(x,u^0(s,x))}{a_j(x)}\right)\,dx<1
$$
and
$$
\max_{s,t \in [0,1]}
\max_{m+1 \le j \le n}\sum_{k=1}^m\sum_{l=m+1}^n|r_{jk}r_{kl}|\exp\int_0^1
\left(\frac{\partial_{u_j}b_j(x,u^0(t,x))}{a_j(x)}-
\frac{\partial_{u_k}b_k(x,u^0(s,x))}{a_k(x)}
\right)\,dx<1.
$$
If one of these two conditions is satisfied, then the linearization of the system of partial integral equations, corresponding to \reff{nnsys}, is Fredholm of index zero (cf. \cite{KR4}).
\end{remark}

Finally we formulate two remarks which concern generalizations of  the phenomenon of forced frequency locking.
\begin{remark}
\label{twofreq}\rm
If the unforced autonomous evolution equation is equivariant under an action of a compact Lie group and if the solution to the unforced equation is a rotating wave  (relative equilibrium) or even a modulated wave (relative periodic orbit) and if the forcing is also of rotating wave type or of modulated wave type, then again phenomena of forced frequency locking type may appear. For ODEs this is described, e.g. in \cite{C00,Re98,RP96,RS11,RS12,SR05}.
For applications of those phenomena to the behavior of optical devices (self-pulsating semiconductor lasers)  see \cite{BandSand,petsand,Rad}.
\end{remark}

\begin{remark}
\label{forcedHopf}\rm
In many applications of forced frequency locking the $T^0$-periodic solution to the unforced autonomous equation is born in a Hopf bifurcation from a stationary solution, and it is natural to ask how this Hopf bifurcation interacts with small $T$-periodic forcing for $T\approx T^0$. In ODE cases this is done, e.g. in \cite{Gam,Scheurle,Zhang}.
\end{remark}

The remaining part of this paper is organized as follows:

Section \ref{proofssys} concerns the first-order system \reff{sys}, and Section \ref{proofeq} the second-order equation \reff{eq}.

In Subsection \ref{Char} we show by means of integration along characteristics that first-oder PDE system \reff{sys} is equivalent to system \reff{parint} of partial integral equations, we discuss the advantages and disadvantages of \reff{sys} and \reff{parint},
and we write \reff{parint} as abstract nonlinear operator equation \reff{abstract} in the function space $C_{per}(\R\times [0,1];\R^2)$.

In Subsections \ref{Invertibility} and \ref{Fredholmness} we show that the
linearization in $\ve=0$, $T=1$ and $u=u^0$ of the abstract nonlinear operator equation \reff{abstract} is Fredholm of index zero in $C_{per}(\R\times [0,1];\R^2)$. In Subsection \ref{kerim} we describe the kernel and the image of this Fredholm operator. This is needed for the Liapunov-Schmidt-like procedure used later.

In Subsection \ref{Regularity} we use the abstract solution regularity result Theorem \ref{app1} of E. N. Dancer in order to show that the solution $u^0$ to the unforced problem has some additional regularity in time (and, hence, in space also).

In Subsection \ref{Ansatz} we introduce new coordinates in a tubular neighborhood of the set $O(u^0)$ of all phase shifts of $u^0$ according to the A. Vanderbauwhede's Theorem \ref{appl2}.

In Subsection \ref{Apriori}
we show that all solutions to
the abstract nonlinear operator equation \reff{abstract} with $\ve \approx 0$, $T\approx 1$ and $u \approx O(u^0)$ can be scaled as $|T-1|\sim \ve$ and dist\,$(u;O(u^0))\sim \ve$.
This is used in Subsection \ref{Proofsys1} in order to derive the scaled abstract nonlinear operator equation \reff{abstract2}. By means of this equation Theorem \ref{sys1} can be proved (using the Implicit Function Theorem 
\ref{app3}) as well as Theorem \ref{sys2}. In particular, formula \reff{Phidef} for the phase equation will be derived.

In Section \ref{proofeq} the rigorous way from problem \reff{eq} to  formula \reff{Phidefeq} of the phase equation is even one step longer: From the second-oder equation \reff{eq} to  first-order system \reff{FOS} (this is done in Subsection \ref{Transfo}), then to  
nonlinear operator equation \reff{abstracteq} and finally 
to scaled nonlinear operator equation \reff{epsnulleq1}. Therefore, in Subsection \ref{formal}
 we show  how formula \reff{Phidefeq} of the phase equation can be derived  directly from 
second-oder equation \reff{eq} by formal calculations.

Finally, in Appendices \ref{appendix1} and \ref{appendix2} we present Theorems \ref{app1}, \ref{appl2} and \ref{app3} from abstract  nonlinear analysis.

\section{Proofs for first-order systems}
\label{proofssys}
\setcounter{equation}{0}
\setcounter{theorem}{0}
In this section we will prove  Theorems \ref{sys1} and \ref{sys2}. Hence, we will suppose that all assumptions of these theorems are satisfied. For the sake of clearness, in all lemmas we will list the assumptions which are  used in the lemma.

We will work with the function space
\begin{eqnarray*}
&&C_{per}(\R\times[0,1];\R^2)\\
&&:=\{u \in C(\R\times[0,1];\R^2):\; u(t+1,x)=u(t,x) \mbox{ for all } (t,x) \in \R \times [0,1]\},
\end{eqnarray*}
which is equipped and complete with the maximum norm $\|\cdot\|_\infty$.

\subsection{Integration along characteristics}
\label{Char}
In this subsection we will show that  problem \reff{sys} is equivalent to a system of two partial integral equations. By tactical reasons, which will  be useful later, we add artificial linear terms $\beta_j(t,x)
u_j(t,x)$ to the left-hand and the right-hand sides of the PDEs in \reff{sys}, i.e.
\begin{eqnarray}
\label{arti}
&&\frac{1}{T}\partial_t u_j(t,x)+a_j(x)\partial_x u_j(t,x)+\beta_j(t,x)u_j(t,x)\nonumber\\
&&=\beta_j(t,x)u_j(t,x)-b_j(x,u(t,x))+\ve f_j(t,x).
\end{eqnarray}
We show that the differential operator on the left-hand side of \reff{arti} has a right inverse, which is a partial integral operator. This procedure is well-known as "integration along characteristics".

In order to simplify notation we introduce the nonlinear superposition operator $B$ from 
$C_{per}(\R\times[0,1];\R^2)^2$ into $C_{per}(\R\times[0,1];\R^2)$,
 defined by
\be
\label{Bdef}
[B_j(\beta,u)](t,x):=\beta_j(t,x)u_j(t,x)-b_j(x,u(t,x)).
\ee
Because of assumption $b_j \in C^3([0,1]\times \R^3)$ (cf. \reff{aass}), the superposition operator $B$ is $C^3$-smooth.
The right-hand side of \reff{arti} is the value of
the function $B_j(\beta,u)+\ve f_j$ in the point $(t,x)$. 

Further, for $j=1,2$, $t \in \R$, $x,y \in [0,1]$, $T>0$ and $\beta \in  
C_{per}(\R\times[0,1];\R^2)$, we write
\be
\label{aldef}
 c_j(t,x,y,T,\beta):=\exp\int_x^y
\frac{\beta_j(t+\al_j(x,z)/T,z)}{a_j(z)}\,dz,
\ee
where $\al_j(x,z)$ is defined in \reff{bjk}.
Then the system of partial integral equations, for $(t,x) \in \R \times [0,1]$), reads
\be
\label{parint}
\left.
\begin{array}{l}
\displaystyle u_1(t,x)-c_1(t,x,0,T,\beta)[r_1u_2(s,0)+\ve g_1(s)]_{s=t+\al_1(x,0)/T}\\
=\displaystyle \int_0^x\frac{c_1(t,x,y,T,\beta)}{a_1(y)}[B_1(\beta,u)+\ve f_1](t+\al_1(x,y)/T,y)\,dy,\\
\displaystyle u_2(t,x)-c_2(t,x,1,T,\beta)[r_2u_1(s,1)+\ve g_2(s)]_{s=t+\al_2(x,1)/T}\\
=\displaystyle -\int_x^1\frac{c_2(t,x,y,T,\beta)}{a_2(y)}[B_2(\beta,u)+\ve f_2](t+\al_2(x,y)/T,y)\,dy.
\end{array}
\right\}
\ee

\begin{lemma}
\label{equiv}
Suppose \reff{aass} and \reff{fass}. Then
for all $\ve>0$, $T>0$ and $\beta \in C_{per}(\R\times[0,1];\R^2)$ the following is true:

(i) Any solution to \reff{sys} is a solution to \reff{parint}.

(ii) If $u \in C_{per}(\R\times[0,1];\R^2)$ is a solution to \reff{parint} and if the partial derivatives $\partial_tu$ and $\partial_t\beta$ exist and are continuous, then $u$ is $C^1$-smooth and is a solution to \reff{sys}.
\end{lemma}
\begin{proof}
Let 
$\ve>0$, $T>0$ and $\beta \in C_{per}(\R\times[0,1];\R^2)$ 
be fixed.

$(i)$ Let a $C^1$-function 
$u:\R\times[0,1]\to\R^2$  be given. Because of $\al_j(x,x)=0$ and $c_j(t,x,x,T,\beta)=1$, we have 
\begin{eqnarray*}
&&u_1(t,x)-c_1(t,x,0,T,\beta)u_1(t+\al_1(x,0)/T,0)\\
&&=\int_0^x\partial_y[c_1(t,x,y,T,\beta)u_1(t+\al_1(x,y)/T,y)]\,dy\\
&&=\int_0^x\partial_yc_1(t,x,y,T,\beta)u_1(t+\al_1(x,y)/T,y)\,dy\\
&&+\int_0^xc_1(t,x,y,T,\beta)[\partial_tu_1(s,y),y)\partial_y\al_1(x,y)/T+\partial_xu_1(s,y)]_{s=t+\al_1(x,y)/T}\,dy.
\end{eqnarray*}
From $\partial_yc_1(t,x,y,T,\beta)=\beta_1(t+\al_1(x,y)/T,y)c_1(t,x,y,T,\beta)/a_1(y)$ and $\partial_y\al_1(x,y)=1/a_1(y)$ it follows that
\begin{eqnarray*}
&&u_1(t,x)-c_1(t,x,0,T,\beta)u_1(t+\al_1(x,0)/T,0)\\
&&=\int_0^x\frac{c_1(t,x,y,T,\beta)}{a_1(y)}\left[\frac{1}{T}\partial_tu_1(s,y)+a_1(y)\partial_xu_1(s,y)+\beta_1(s,y)u_1(s,y)\right]_{s=t+\al_1(x,y)/T}\,dy.
\end{eqnarray*}
Hence, if $u$ is a solution to \reff{sys}, i.e. to \reff{arti}, then it solves the first equation in \reff{parint}. Similarly one shows that it solves the second equation in \reff{parint}. 

$(ii)$ Let $u \in  C_{per}(\R\times[0,1];\R^2)$ be a solution to \reff{parint}. Then the first equation in \reff{parint} with $x=0$ yields
$$
u_1(t,0)=c_1(t,0,0,T,\beta)[r_1u_2(s,0)+\ve g_1(s)]_{s=t+\al_1(0,0)/T}=r_1u_2(t,0)+\ve g_1(t),
$$
i.e. the first boundary condition in \reff{sys} is satisfied. Similarly one shows that second boundary condition in \reff{sys} is satisfied also. 

Moreover, if $\partial_tu$ and $\partial_t\beta$ exist and are continuous, then it follows directly from \reff{parint}
that also $\partial_xu$ exists and and is continuous, i.e. $u$ is $C^1$-smooth.

Now let us show that the first differential equation in \reff{sys} is satisfied. For that we use that
\be
\label{chara}
\left(\frac{1}{T}\partial_t+a_1(x)\partial_x\right)\phi(t+\al_1(x,y)/T,y)=0 \mbox{ for all } \phi \in C^1(\R).
\ee
Therefore, \reff{aldef} yields
\be
\label{chara1}
\left(\frac{1}{T}\partial_t+a_1(x)\partial_x\right)
c_1(t,x,y,T,\beta)=-\beta_1(t,x)c_1(t,x,y,T,\beta).
\ee
Applying the differential operator $\frac{1}{T}\partial_t+a_1(x)\partial_x$ to the first equation in \reff{parint} and using \reff{chara} and \reff{chara1}, we get
\begin{eqnarray*}
&&\frac{1}{T}\partial_tu_1(t,x)+a_1(x)\partial_xu_1(t,x)
+\beta_1(t,x)c_1(t,x,0,T,\beta)[r_1u_2(s,0)+\ve g_1(s)]_{s=t+\al_1(x,0)/T}\\
&&=[\ve f_1+B_1(\beta,u)](t,x)-\beta_1(t,x)\int_0^x\frac{c_1(t,x,y,T,\beta)}{a_1(y)}[\ve f_1+
B_1(\beta,u)](t+\al_1(x,y)/T,y)\,dy.
\end{eqnarray*}
Using the first equation in \reff{parint} again, we get the first equation in \reff{arti}, i.e. the first equation in \reff{sys}. Similarly one shows that also the second equation in \reff{sys} is satisfied.
\end{proof}

\begin{remark}
\label{advant}\rm
Roughly speaking, the advantages and disadvantages of systems \reff{sys} and \reff{parint} are the following: System \reff{sys} depends smoothly on the control parameter $T$ uniformly in  the state parameter $u$. This is not the case for system \reff{parint} because, if one differentiates equations in \reff{parint} with respect to $T$, then $u$ loses one degree of differentiability with respect to $t$.
On the other hand, the linearizations of \reff{parint} with respect to $u$ are Fredholm of index zero on a $T$-independent function space (see Subsection \ref{Fredholmness}), what is not the case for \reff{sys}.
\end{remark}

Let us write system \reff{parint} in an  abstract form. In order to do so, we introduce the function space
$$
C_{per}(\R;\R^2):=\{u \in C(\R;\R^2):\; u(t+1,x)=u(t,x) \mbox{ for all } (t,x) \in \R \times [0,1]\},
$$
which is equipped and complete with the maximum norm, which will be denoted by $\|\cdot\|_\infty$  again, i.e.
$\|w\|_\infty:=\max\{|w_j(t)|: \; j=1,2,\; t\in\R\}$ for $w \in C_{per}(\R;\R^2)$.
Further, we introduce the linear bounded operator $R:C_{per}(\R\times[0,1];\R^2)\to C_{per}(\R;\R^2)$,  defined by 
\be 
\label{Rdef}
[R(u_1,u_2)](t):=(r_1u_2(t,0),r_2u_1(t,1)).
\ee
Finally, for $T>0$ and $\beta \in C_{per}(\R\times[0,1];\R^2)$, we introduce linear bounded operators
$C(T,\beta):C_{per}(\R;\R^2)\to C_{per}(\R\times[0,1];\R^2)$
and $D(T,\beta):C_{per}(\R\times [0,1];\R^2)\to C_{per}(\R\times[0,1];\R^2)$,  defined by
\be
\label{Cdef}
\left.
\begin{array}{l}
~[C_1(T,\beta)w](t,x):=c_1(t,x,0,T,\beta)w_1(t+\al_1(x,0)/T),\\
~[C_2(T,\beta)w](t,x):=c_2(t,x,1,T,\beta)w_2(t+\al_2(x,1)/T)
\end{array}
\right\}
\ee
and
\be
\label{Ddef}
\left.
\begin{array}{l}
~\displaystyle[D_1(T,\beta)u](t,x):=\int_0^x\frac{c_1(t,x,y,T,\beta)}{a_1(y)}u_1(t+\al_1(x,y)/T,y)\,dy,\\
~\displaystyle[D_2(T,\beta)u](t,x):=-\int_x^1\frac{c_2(t,x,y,T,\beta)}{a_2(y)}u_2(t+\al_2(x,y)/T,y)\,dy.
\end{array}
\right\}
\ee
Then problem \reff{parint} can be written as
\be
\label{abstract}
u=C(T,\beta)(Ru+\ve g) +D(T,\beta)(B(\beta,u)+\ve f).
\ee
Note that the variable $\beta$ in \reff{abstract} is still artificial, i.e.
if $u$ satisfies \reff{abstract}, then $u$ satisfies \reff{abstract} with $\beta$ replaced by any other $\widetilde\beta\in C_{per}(\R\times[0,1];\R^2)$ also.
Therefore, we have the following:
\begin{eqnarray}
\label{artif}
&&\mbox{If \reff{abstract} is satisfied, then}\nonumber\\
&&\partial_\beta C(T,\beta)(Ru+\ve g)+\partial_\beta D(T,\beta)(B(\beta,u)+\ve f)
+D(T,\beta)\partial_\beta B(\beta,u)=0.
\end{eqnarray}

From the definitions \reff{aldef}, \reff{Cdef} and \reff{Ddef} it follows that the maps $\beta \mapsto C(T,\beta)R$ and 
$\beta \mapsto D(T,\beta)$ are $C^\infty$-smooth from $C_{per}(\R\times[0,1];\R^2)$ into
${\cal L}(C_{per}(\R\times[0,1];\R^2))$ (with respect to the uniform operator norm).
Similarly, from \reff{Bdef} and from the smoothness assumption
$b_j \in C^3([0,1]\times \R)$ (cf. \reff{aass})
it follows that the map $(\beta,u)\mapsto B(\beta,u)$ is $C^2$-smooth from
$C_{per}(\R\times[0,1];\R^2)^2$ into $C_{per}(\R\times[0,1];\R^2)$.

Unfortunately, the maps $T \mapsto C(T,\beta)R$ and $T \mapsto D(T,\beta)$ are not continuous
from $(0,\infty)$ into
${\cal L}(C_{per}(\R\times[0,1];\R^2))$
(with respect to the uniform operator norm).
This causes some technical difficulties in the analysis of equation \reff{abstract}. But we have 
\be 
\label{besch}
\sup\{\|C(T,\beta)Ru\|_\infty+\|D(T,\beta)u\|_\infty:\; T>0,\; \|\beta\|_\infty\le c, \;\|u\|_\infty\le 1\}<\infty 
\ee
for any $c>0$ and, moreover, we have
\be 
\label{conti}
(T,\beta,u)\mapsto (C(T,\beta)Ru,D(T,\beta)u)
\mbox{ is continuous} 
\ee
from $(0,\infty)\times
C_{per}(\R\times[0,1];\R^2)^2$ into 
$C_{per}(\R\times[0,1];\R^2)^2$.

\subsection{Invertibility of \boldmath$I-C(T,\beta^0)R$\unboldmath}
\label{Invertibility}
In what follows, we introduce a function $\beta^0 \in C_{per}(\R\times [0,1];\R^2)$  by (cf. \reff{bjk})
\be 
\label{betanulldef}
\beta^0(t,x):=(b_{11}(t,x),b_{22}(t,x))=\left(\partial_{u_1}b_1(x,u^0(t,x)),
\partial_{u_2}b_2(x,u^0(t,x))\right).
\ee
\begin{lemma}\label{Invert}
Suppose \reff{aass}, \reff{aassnot}  and \reff{unull}, and assume that one of the conditions \reff{nonressys}
and \reff{nonressys1} is satisfied. Then
there exists $\ve_0 \in (0,1)$ such that for all 
$T$ with $|T-1|<\ve_0$ and all $f \in C_{per}(\R\times [0,1];\R^2)$ 
there exists a unique solution $u=u_{T,f} \in 
C_{per}(\R\times [0,1];\R^2)$ to the equation $u=C(T,\beta^0)Ru+f$, and, moreover,
$$
\sup\left\{\|u_{T,f}\|_\infty:\; |T-1|<\ve_0,\;
  f \in C_{per}(\R\times [0,1];\R^2),\; \|f\|_\infty \le 1  \right\}<\infty.
$$
\end{lemma}
\begin{proof}
Because of $\beta=\beta^0$ is fixed, we will skip the dependence on $\beta$ in the notation
$$
C(T):=C(T,\beta^0) \mbox{ and } c_j(t,x,y,T):=c_j(t,x,y,T,\beta^0).
$$
Remark that the map $T \mapsto C(T)$ is not uniformly continuous, but only strongly continuous. Hence, for proving the invertibility of the operator $I-C(T)R$ for $T\approx 1$ it is not sufficient to prove the invertibility of the operator $I-C(1)R$.

Take $f \in C_{per}(\R \times [0,1];\R^2)$. We have to show that  for $T \approx 1$ there exists a unique solution $u \in C_{per}(\R \times [0,1];\R^2)$ to the equation
  \be
  \label{1a}
  (I-C(T)R)u=f
  \ee
and that $\|u\|_\infty \le \mbox{const}\|f\|_\infty$, 
where the constant does not depend on $T$ and $f$.
  Equation \reff{1a} is satisfied if and only if for 
  all $t \in \R$ and $x \in [0,1]$ we have
  \begin{eqnarray}
    \label{2a}
    u_1(t,x)&=&r_1c_1(t,x,0,T)u_2(t+\al_1(x,0)/T,0)+f_1(t,x),\\
     \label{3a}
    u_2(t,x)&=&r_2c_2(t,x,1,T)u_1(t+\al_2(x,1)/T,1)+f_2(t,x).
    \end{eqnarray}
 System \reff{2a}--\reff{3a} is satisfied if and only 
 if \reff{2a} is true and if 
  \begin{eqnarray}
     \label{4a}
    u_2(t,x)&=&r_1r_2c_1(t+\al_2(x,1)/T,1,0,T)
    c_2(t,x,1,T)u_2(t+(\al_1(1,0)+\al_2(x,1))/T,0)
    \nonumber\\
    &&+r_2c_2(t,x,1,T)f_1(t+\al_2(x,1)/T,x)+
    f_2(t,x).
  \end{eqnarray}
 Put $x=0$ in \reff{4a} and get
   \begin{eqnarray}
     \label{5a}
    u_2(t,0)&=&r_1r_2c_1(t+\al_2(0,1)/T,1,0,T)
    c_2(t,0,1,T)u_2(t+(\al_1(1,0)+\al_2(0,1))/T,0)
    \nonumber\\
    &&+r_2c_2(t,x,1,T)f_1(t+\al_2(0,1)/T,0)+f_2(t,0).
   \end{eqnarray}
Similarly, system \reff{2a}--\reff{3a} 
is satisfied if and only 
 if \reff{3a} is true and 
  \begin{eqnarray}
     \label{4b}
    u_1(t,x)&=&r_1r_2c_2(t+\al_1(x,0)/T,0,1,T)
    c_1(t,x,0,T)u_1(t+(\al_2(0,1)+\al_1(x,0))/T,0)
    \nonumber\\
    &&+r_2c_2(t+\al_1(x,0)/T,0,1,T)f_2(t+\al_1(x,0)/T,x))+
    f_1(t,x),
  \end{eqnarray}
  i.e. if and only if  \reff{3a} and \reff{4b} are true and 
   \begin{eqnarray}
     \label{5b}
    u_2(t,0)&=&r_1r_2c_1(t+\al_2(0,1)/T,1,0,T)
    c_2(t,0,1,T)u_2(t+(\al_1(1,0)+\al_2(0,1))/T,0)
    \nonumber\\
    &&+r_1c_1(t+\al_2(0,1)/T,1,0,T)f_1(t+\al_2(0,1)/T,0))+f_2(t,0).
   \end{eqnarray}
   
Let us consider equation \reff{5a}. It is a functional equation for the unknown function $u_2(\cdot,0)$. In order to solve it, let us denote
   by $C_{per}(\R)$ the Banach space of all $1$-periodic continuous functions $\tilde{u}:\R \to\R$ with the norm
   $\|\tilde{u}\|_\infty:=\max\{|\tilde{u}(t)|: \; t \in \R\}$.  Equation \reff{5a} is an equation in  $C_{per}(\R)$ of the type
   \be
   \label{6a}
   (I-\widetilde{C}(T))\tilde{u}=\tilde{f}(T)
   \ee
   with $\tilde{u},\tilde{f}(T)\in C_{per}(\R)$ defined by $\tilde{u}(t):=u_2(t,0)$ and
   $$
   [\tilde{f}(T)](t):=r_1c_1(t+\al_2(0,1)/T,1,0,T)
   f_1(t+\al_2(0,1)/T,0))+f_2(t,0)
   $$
   and with $\widetilde{C}(T)\in {\cal L}(C_{per}(\R))$ defined by
   $$
   [\widetilde{C}(T)\tilde{u}](t):=
 r_1r_2c_1(t+\al_2(0,1)/T,1,0,T)
    c_2(t,0,1,T)u_2(t+(\al_1(1,0)+\al_2(0,1))/T,0).
   $$
   From the definitions of the functions $c_1$ 
   and $c_2$ (cf.  \reff{aldef}) it follows that
   \begin{eqnarray*}
   &&c_1(t+\al_2(0,1)/T,1,0,T)
    c_2(t,0,1,T)\\
   &&=\exp \int_0^1\left(\frac{b_{22}(t+\al_2(0,x)/T,x)}{a_2(x)}-
 \frac{b_{11}(t+(\al_1(1,x)+\al_2(0,1))/T,x)}{a_1(x)}\right)\,dx
\end{eqnarray*}
 and, hence,
\begin{eqnarray*}
   &&c_1(s+\al_2(0,1)/T,1,0,T)
    c_2(s,0,1,T)|_{T=1,s=t-\al_2(0,1)}\\
   &&=\exp \int_0^1\left(\frac{b_{22}(t-\al_2(x,1),x)}{a_2(x)}-
 \frac{b_{11}(t-\al_1(x,1),x)}{a_1(x)}\right)\,dx,  
\end{eqnarray*}
Consequently, if assumption \reff{nonressys} is satisfied, then
$|r_1r_2c_1(t+\al_2(0,1),1,0,1)
    c_2(t,0,1,1)|\not=1$ for all $t \in \R$.

First, let us consider the case that 
$$
c_+:=\max\{|r_1r_2c_1(t+\al_2(0,1),1,0,1)
    c_2(t,0,1,1)|:\; t \in \R\}<1.
$$
 Then there exists $\ve_0 \in (0,1)$ such that for $|T-1|<\ve_0$
   we have
   $$|r_1r_2c_1(t+\al_2(0,1)/T,1,0,T)
    c_2(t,0,1,T)|
   \le \frac{1+c_+}{2}<1 \quad\mbox{for all } t \in \R.$$
   Therefore,
   $$
   \|\widetilde{C}(T)\|_{ {\cal L}(C_{per}(\R))}\le \frac{1+c_+}{2}<1 \quad\mbox{for  }  |T-1|<\ve_0.
   $$
   Hence, for $|T-1|<\ve_0$ the operator $I-\widetilde{C}(T)$ is an isomorphism on $C_{per}(\R)$, and, moreover,
   $$
   \|(I-\widetilde{C}(T))^{-1}\|_{ {\cal L}(C_{2\pi}(\R))}\le \frac{1}{1-\frac{1+c_+}{2}}=\frac{2}{1-c_+}.
   $$
   Therefore, for $|T-1|<\ve_0$
   there exists a unique solution $u_2(\cdot,0)\in C_{per}(\R)$
   to \reff{5a}, and
   $
   \|u_2(\cdot,0)\|_\infty \le \mbox{const}\|\tilde{f}(T)\|_\infty\le \mbox{const}\|f\|_\infty,
   $
   where the constants do not depend on $T$ and $f$.
   Inserting this solution into the right-hand sides of \reff{4a} 
   and \reff{2a}, we get the unique solution $u=(u_1,u_2) \in  C_{per}(\R\times [0,1];\R^2)$ to \reff{2a}--\reff{3a}.
   Moreover, this solution satisfies the a priori estimate
   $\|u\|_\infty \le \mbox{const}\|f\|_\infty$,   where the constant does not depend on $T$ and $f$.

    Now, let us consider the case that 
 $$
c_-:=\min\{|r_1r_2c_1(t+\al_2(0,1),1,0,1)
    c_2(t,0,1,1)|:\; t \in \R\}>1.
$$   
Then there exists $\ve_0>0$ such that for all  $|T-1|<\ve_0$
   we have
   $$|r_1r_2c_1(t+\al_2(0,1)/T,1,0,T)
    c_2(t,0,1,T)|
   \ge \frac{1+c_-}{2}\ge 1 \quad\mbox{ for all } t \in \R.$$ 
Equation \reff{5a} is equivalent to
 \begin{eqnarray*}
      &&u_2(t,0)=
    \frac{u_2(t-(\al_1(1,0)+\al_2(0,1))/T,0)}
    {r_1r_2c_1(t-\al_1(1,0)/T,1,0,T)
    c_2(t-(\al_1(1,0)+\al_2(0,1))/T,0,1,T)}\\
    &&+\frac{f_1(t-\al_1(1,0)/T,0)}{r_2c_1(t-(\al_1(1,0)+\al_2(0,1))/T,1,0,T)}\\
    &&+\frac{f_2(t-(\al_1(1,0)+\al_2(0,1))/T,0)}{r_1r_2c_1(t-\al_1(1,0)/T,1,0,T)
    c_2(t-(\al_1(1,0)+\al_2(0,1))/T,0,1,T)}.
 \end{eqnarray*}
 This equation is of the type \reff{6a} again, but now with 
 $$
 \|\widetilde{C}(T)\|_{ {\cal L}(C_{per}(\R))}\le\frac{2}{1+c_-}<1.
 $$
Hence, we can proceed as above.

Similarly one deals with the case if condition \reff{nonressys1} is satisfied.  Then equation \reff{5b} is uniquely solvable with corresponding estimate, and so is equation \reff{4b} and, hence, system \reff{2a}--\reff{3a}.
\end{proof}

\subsection{Fredholmness of \boldmath$I-C(T,\beta^0)R-D(T,\beta^0)\partial_uB(\beta^0,u^0)$\unboldmath}
\label{Fredholmness}
In this subsection we prove the following Fredholmness result.
\begin{lemma}
\label{Fredholm}
Suppose \reff{aass}, \reff{aassnot} and \reff{unull}, and assume that  one of the conditions \reff{nonressys}
and \reff{nonressys1} is satisfied. Take $\ve_0>0$ from Lemma \ref{Invert}. Then, for all 
$T$ with $|T-1|<\ve_0$, the operator $I-C(T,\beta^0)R-D(T,\beta^0)\partial_uB(\beta^0,u^0)$
is Fredholm of index zero from $C_{per}(\R\times [0,1];\R^2)$ into itself.
\end{lemma}
\begin{proof}
	Because of $T$ and  $\beta=\beta^0$ are fixed, in the notations below we will not mention the dependence on $T$ and $\beta$, i.e. 
\be
\label{DTdef}
C:=C(T,\beta^0),\; D:=D(T,\beta^0),\; B^0:=\partial_u B(\beta^0,u^0)
\ee
and (cf. the definitions of the functions $c_1$,$c_2$ and $\beta^0$ in \reff{aldef} and \reff{betanulldef})
$$
c_j(t,x,y):=c_j(t,x,y,T,\beta^0)
=\exp\int_x^y\frac{b_{jj}(t+\al_j(x,z)/T,z)}{a_j(z)}\,dz.
$$

On account of  Lemma~\ref{Invert},  the operator $I-CR-DB^0$ is Fredholm of index zero if and only if
 the operator $I-(I-CR)^{-1}DB^0$ is Fredholm of index zero. 
Hence, it suffices to show that  the operator $((I-CR)^{-1}DB^0)^2$ is compact. This is a consequence of the following Fredholmness criterion of S. M. Nikolskii (cf. e.g.  \cite[Theorem XIII.5.2]{Kant}):
   If a Banach space $U$ and a linear bounded operator $K:U\to U$ are given such that $K^2$ is compact, then
   the operator $I-K$ is Fredholm of index zero.
   
A straightforward calculation shows that
\be\label{prod}
((I-CR)^{-1}DB^0)^2
=(I-CR)^{-1}
\left((DB^0)^2+DB^0CR(I-CR)^{-1}DB^0\right).
\ee
Using this and Lemma \ref{Invert} again, we get that it suffices to show that the operators $DB^0D$ and $DB^0CR$ are compact.

Let us show that $DB^0D$ is compact from $C_{per}(\R\times [0,1];\R^2)$ into itself. Here the key point is, that the partial integral operator $D$ is not compact, but the "full" integral operator $DB^0D$ is compact.

Take $u \in C_{per}(\R\times [0,1];\R^2)$. From \reff{Bdef}
 and \reff{DTdef} it follows that 
\be 
\label{anti}
[B^0u](t,x)=-(b_{12}(t,x)u_2(t,x),b_{21}(t,x)u_1(t,x)).
\ee
Hence, the definitions \reff{Ddef} and \reff{DTdef} imply that the first component of 
$[DB^0u](t,x)$ is
\be
\label{D1B}
[D_1B^0u](t,x)=-\int_0^x\frac{c_1(t,x,y)}{a_1(x)}
\left[b_{12}(s,y)u_2(s,y)\right]_{s=t+\al_1(x,y)/T}\,dy.
\ee
Using  \reff{Ddef} and \reff{DTdef}, we get
 \be
\label{D1BD}
[D_1B^0Du](t,x)=\int_0^x\int_y^1d(t,x,y,z) u_2(t+(\al_2(x,y)+\al_1(y,z))/T,z)\,dzdy,
\ee
where
$$
d(t,x,y,z):=-\frac{c_1(t,x,y)}{a_1(y)a_2(z)}
\left[b_{12}(s,y)c_2(s,z)\right]_{s=t+\al_1(x,y)/T}.
$$
Now we change the order of integration in \reff{D1BD} according to
\be
\label{change}
\int_0^xdy\int_y^1dz=\int_0^xdz\int_0^zdy+\int_x^1dz\int_0^xdy.
\ee

Let us consider the first summand in the right-hand side of \reff{change}, it is
\be
\label{firsts}
{\cal I}(t,x,u):=\int_0^x\int_0^zd(t,x,y,z) u_2(t+(\al_2(x,y)+\al_1(y,z))/T,z)\,dydz.
\ee
In the inner integral in \reff{firsts} we replace the integration variable $y$ by  $\eta$ according to
$$
\eta=\widehat{\eta}(t,x,y,z):=t+(\al_2(x,y)+\al_1(y,z))/T,z)=t+\frac{1}{T}\left(\int_x^y\frac{d\xi}{a_2(\xi)}+\int_y^z\frac{d\xi}{a_1(\xi)}\right).
$$
Here we used the definitions \reff{bjk} of $\al_j$. Because of assumption $a_1(x)\not=a_2(x)$ for all $x \in [0,1]$ (cf. \reff{aass}) we have
$$
\partial_y\widehat{\eta}(t,x,y,z)=\frac{1}{T}\left(\frac{1}{a_2(y)}-\frac{1}{a_1(y)}\right)\not=0
\mbox{ for all } y \in [0,1],
$$
i.e. the function $y \mapsto \widehat{\eta}(t,x,y,z)$ is strictly monotone. Let us denote its inverse function by $\eta \mapsto \widehat{y}(t,x,\eta,z)$. Then
$$
\partial_\eta \widehat{y}(t,x,\eta,z)=
\left(\partial_y\widehat{\eta}(t,x,\widehat{y}(t,x,\eta,z),z)\right)^{-1}=
T\left(\frac{1}{a_2(\widehat{y}(t,x,\eta,z))}-\frac{1}{a_1(\widehat{y}(t,x,\eta,z))}\right)^{-1},
$$
and 
$$
{\cal I}(t,x,u)=\int_0^x\int_{\widehat{\eta}(t,x,0,z)}^{\widehat{\eta}(t,x,z,z)}
d(t,x,\widehat{y}(t,x,\eta,z),z)u_2(\eta,z)
\partial_\eta \widehat{y}(t,x,\eta,z)d\eta \,dz.
$$
The smoothness assumptions in \reff{aass} and \reff{kersys} on the data $a_j$, $b_j$ and $u^0$ yield that the functions $d$, $\widehat{\eta}$, $\widehat{y}$ and $\partial_\eta \widehat{y}$
are $C^1$-smooth with bounded partial derivatives. Therefore,
$$
|{\cal I}(t,x,u)|+|\partial_t{\cal I}(t,x,u)|+|\partial_x{\cal I}(t,x,u)|\le \mbox{const}\,\|u\|_\infty,
$$
where the constant does not depend on $t$, $x$ and $u$.

Similarly one shows that also the second summand in the right-hand side of \reff{change}, which is
$$
\int_x^1\int_0^xd(t,x,y,z) u_2(t+(\al_2(x,y)+\al_1(y,z))/T,z)\,dydz,
$$ 
is $C^1$-smooth with respect to $t$ and $x$ and that its partial derivatives can be estimated by $\mbox{const}\,\|u\|_\infty$,
where the constant does not depend on $t$, $x$ and $u$. Similarly one shows that also $D_2B^0Du$ is $C^1$-smooth with respect to $t$ and $x$ and that its partial derivatives can be estimated by $\mbox{const}\,\|u\|_\infty$, where the constant does not depend on $t$, $x$ and $u$. This way we get
$
\|\partial_t DB^0Du\|_\infty+\|\partial_x DB^0Du\|_\infty \le \mbox{const}\,\|u\|_\infty,
$
where the constant does not depend on $u$.
Hence, the Arzela-Ascoli Theorem yields that the operator $DB^0D$ is compact from $C_{per}(\R\times [0,1];\R^2)$ into itself.

Finally, let us show that the operator $DB^0CR$ 
is compact from $C_{per}(\R\times [0,1];\R^2)$ into itself.
From \reff{Rdef}, \reff{Cdef} and \reff{D1B} it follows that
$$
[D_1B^0CRu](t,x)=\int_0^xd(t,x,y,) u_1(t+(\al_2(x,y)+\al_1(y,1))/T,1)\,dy
$$
where
$$
d(t,x,y):=\frac{r_2c_1(t,x,y)}{a_1(y)}
\left[b_{12}(s,y)c_2(s,z)\right]_{s=t+\al_1(x,y)/T}.
$$
Now we change the integration variable $y$ to $\eta=t+(\al_2(x,y)+\al_1(y,1))/T$ and proceed as above. This way  we get the desired bound
$\|\partial_t DB^0CRu\|_\infty+\|\partial_x DBCRu\|_\infty \le \mbox{const}\,\|u\|_\infty$,
where the constant does not depend on $u$.
\end{proof}

\begin{remark}
\label{changerem}\rm
In the proof above 
one can see the reason for inserting the artificial term $\beta_j(t,x)u_j(t,x)$ into \reff{arti} and for the choice $\beta=\beta^0$ (cf. \reff{betanulldef}):
It leads to the linear operator $B^0$, which has the property that the first component $[B^0_1u](t,x)$ does not depend on $u_1$ and the second component $[B^0_2u](t,x)$ does not depend on $u_2$. This leads to \reff{D1BD} and, hence, to  the change of integration variables $y \mapsto \eta$. If one would not introduce the artificial term, i.e. if one would choose $\beta=0$, then the operator $DB^0D$ would not be compact from $C_{per}(\R\times [0,1];\R^2)$ into itself, in general.
\end{remark}

\subsection{Kernel and image of \boldmath$I-C(1,\beta^0)R-D(1,\beta^0)\partial_uB(\beta^0,u^0)$\unboldmath}
\label{kerim}
In this subsection we describe the kernel and the image of the Fredholm operator $I-C(1,\beta^0)R-D(1,\beta^0)\partial_uB(\beta^0,u^0)$. For that we need some more notation.

We denote by $C^1_{per}(\R\times[0,1];\R^2)$ the space of all $C^1$-smooth  
$u \in C_{per}(\R\times[0,1];\R^2)$. Further, for $T>0$ we consider linear operators $\A(T)$  from $C^1_{per}(\R\times[0,1];\R^2)$ into $C_{per}(\R\times[0,1];\R^2)$ with components 
\be 
\label{Adef}
[\A_j(T)u](t,x):=\frac{1}{T}\partial_tu_j(t,x)+a_j(x)\partial_xu_j(t,x)+
\partial_{u_j}b_j(x,u^0(t,x))u_j(t,x).
\ee
We also consider the linear bounded functional $\phi:C_{per}(\R\times[0,1];\R^2)
\to \R$, which is defined 
for  $u \in C^1_{per}(\R\times[0,1];\R^2)$
by
\begin{eqnarray} 
\label{phidef}
\phi(u)&:=&\int_0^1\int_0^1
\left([\A_1(1)u](t,x)u_1^*(t,x)+[\A_2(1)u](t,x)u_2^*(t,x)\right)\,dtdx\nonumber\\
&&+\int_0^1\left(a_1(0)u_1(t,0)u_1^*(t,0)
-a_2(1)u_2(t,1)u_2^*(t,1)\right)dt.
\end{eqnarray}
Because of $u^*$ is $C^1$-smooth, we have 
$\sup\{|\phi(u)|:\; u \in  C^1_{per}(\R\times[0,1];\R^2), \; \|u\|_\infty\le 1\}<\infty$.
Hence, although \reff{phidef} defines the values of $\phi$ on a dense subspace only, it
really determines a linear bounded functional on $C_{per}(\R\times[0,1];\R^2)$.
\begin{lemma}
\label{lemker}
Suppose \reff{aass}, \reff{aassnot}, \reff{unull}, \reff{kersys} and
\reff{keradsys}, and assume that one of the conditions \reff{nonressys}
and \reff{nonressys1} is satisfied.
Then the following is true:

(i) $\ker\left[I-C(1,\beta^0)R-D(1,\beta^0)\partial_uB(\beta^0,u^0)\right]=\mbox{span}\{\de_tu^0\}$

(ii) $\im\left[I-C(1,\beta^0)R-D(1,\beta^0)\partial_uB(\beta^0,u^0)\right]=\ker \phi$

(iii) $C_{per}(\R\times[0,1];\R^2)
=\ker \phi \oplus
\mbox{span}\{D(1,\beta^0)\de_tu^0\}$
\end{lemma}
\begin{proof} $(i)$
In the proof of Lemma \ref{Fredholm} we showed that, if \reff{aass}, \reff{aassnot}  and \reff{unull} and  one of the conditions \reff{nonressys}
and \reff{nonressys1} are satisfied,
then for any $u \in C_{per}(\R\times [0,1];\R^2)$ the element
$
\left[(I-C(1,\beta^0)R)^{-1}D(1,\beta^0)\partial_uB(\beta^0,u^0)\right]^2u
$
is $C^1$-smooth. Especially, any solution 
$u \in C_{per}(\R\times [0,1];\R^2)$ to
\be 
\label{line}
u-C(1,\beta(u^0))Ru =D(1,\beta^0)\partial_u B(\beta^0,u^0)u,
\ee
is $C^1$-smooth. Therefore, Lemma \ref{equiv}
(with $\ve=0$, $T=1$, $\beta=\beta^0$ and $B(\beta,u)$ replaced by $\partial_uB(\beta(u^0),u^0)u$) yields that any 
solution to \reff{line}
is a solution to \reff{linsys}, and vice versa.
Hence, assumption \reff{kersys} yields
assertion $(i)$ of the lemma.

$(ii)$ In Lemma \ref{equiv} (with $\ve=0$ and with
$\beta$ and $B(\beta,u)$ replaced by $\beta^0$ and $\partial_u(\beta^0,u^0)$, respectively) we have shown that
\be 
\label{AC}
\A(T)C(T,\beta^0)Ru= 
\A(T)D(T,\beta^0)u-u=0
\ee
for all $T>0$ and 
$u \in  C^1_{per}(\R\times[0,1];\R^2)$.
Therefore,
\begin{eqnarray*}
&&\sum_{j=1}^2\int_0^1\int_0^1
[\A_j(1)(I-C(1,\beta^0)R-D(1,\beta^0)\partial_u(\beta^0,u^0))u](t,x)
u^*_j(t,x)\,dtdx\\
&&=\sum_{j=1}^2\int_0^1\int_0^1
[\A_j(1)-\partial_uB_j(\beta^0,u^0))u](t,x)u^*_j(t,x)\,dtdx\\
&&=
\sum_{j,k=1}^2\int_0^1\int_0^1
\left((\de_ch_j(t,x)+a_j(x)\de_cu_j(t,x)+
b_{j}(t,x)u_k(t,x)\right)u^*_j(t,x)\,dtdx\\&&=\sum_{j=1}^2\left[a_j(x)\int_0^1u_j(t,x)u^*_j(t,x)\,dt
\right]_{x=0}^{x=1}\\
&&= \int_0^1(a_1(0)(r_1u_2(t,0)-u_1(t,0))u_1^*(t,0)+a_2(1)(u_2(t,1)-r_2u_1(t,1))u_2^*(t,1))\,dt
\end{eqnarray*}
for all $u \in  C^1_{per}(\R\times[0,1];\R^2)$.
Here, along with the integration by parts, we used the definition \reff{bjk} of the coefficients $b_{j}$, the definitions \reff{Bdef} and \reff{Adef} of the operators $B_j$ and $\A_j(T)$, the properties \reff{anti} 
of the operators $B^0_j=\partial_uB_j(\beta^0,u^0)$ and \reff{AC} of $\A_j(T)$ and the fact
that $u^*$ is a solution to \reff{adsys}.
If we write $v:=(I-C(1,\beta^0)R-D(1,\beta^0)\partial_u(\beta^0,u^0))u$, then $v_1(t,0)=
u_1(t,0)-r_1u_2(t,0)$ and $v_2(t,1)=u_2(t,1)-r_2u_1(t,1)$, and we get 
$$
\sum_{j=1}^2\int_0^1\int_0^1
[\A_j(1)v(t,x)u^*_j(t,x)\,dtdx
=\int_0^1(a_2(1)v_2(t,1)u_2^*(t,1)-
a_1(0)v_1(t,0)u_1^*(t,0))\,dt.
$$
Therefore, $\phi((I-C(1,\beta^0)R-D(1,\beta^0)\partial_u(\beta^0,u^0))u)=0$ for all  $u \in  C^1_{per}(\R\times[0,1];\R^2)$ and, hence,
for all $u \in  C_{per}(\R\times[0,1];\R^2)$.
In other words,
$$
\im(I-C(1,\beta^0)R-D(1,\beta^0)\partial_u(\beta^0,u^0))\subseteq
\ker \phi.
$$
On the other side, it follows from Lemma \ref{Fredholm} and from assertion $(i)$ that 
\begin{eqnarray*}
&&\codim\im(I-C(1,\beta^0)R-D(1,\beta^0)\partial_u(\beta^0,u^0))\\
&&=\dim \ker
(I-C(1,\beta^0)R-D(1,\beta^0)\partial_u(\beta^0,u^0))=1.
\end{eqnarray*}
Hence, for proving assertion $(ii)$, it remains to show that $\codim\ker \phi=1$, i.e. that $\phi\not=0$. In order to show this, we 
do the following calculations: Let $v^0:=D(1,\beta^0)\de_tu^0$.
Then $v^0_1(t,0)=v^0_2(t,1)=0$ (cf. \reff{AC}) and $\A v^0=\de_tu^0$.
Hence,
\be  
\label{ortho}
\phi(D(1,\beta^0)\de_tu^0)
=\int_0^1\int_0^1
(\de_tu^0_1u_1^*+\de_tu_2^0u_2^*)\,dtdx=1.
\ee
Here we used assumption \reff{keradsys}.
Finally, assertion $(iii)$ follows from \reff{ortho}.
\end{proof}

\begin{remark}
\label{P}\rm
It follows from \reff{ortho} that the projection, corresponding to the algebraic sum in assertion $(iii)$ of Lemma \ref{lemker}, is
$Pu:=\phi(u)D(1,\beta^0)\de_tu^0$, i.e.
$$
\ker P=\im(I-C(1,\beta^0)R-D(1,\beta^0)\partial_uB(\beta^0,u^0)),\; \im P=\mbox{span}\{D(1,\beta^0)\de_tu^0\}.
$$
\end{remark}

\subsection{Additional regularity of \boldmath$u^0$\unboldmath}
\label{Regularity}
In this subsection we will show that, if conditions  
\reff{aass}, \reff{aassnot} and \reff{unull}  and  one of the conditions \reff{nonressys} and \reff{nonressys1} are satisfied, then
not only the first partial derivatives $\partial_tu^0$ and $\partial_xu^0$ exist and are continuous (this is by assumption \reff{unull}), but also the second partial derivative $\partial_t^2u^0$ exists and is continuous. For that we use Theorem \ref{app1}.

Let us introduce the setting of Theorem \ref{app1} as follows: The Banach space $U$ is the function space $C_{per}(\R \times [0,1];\R^2)$ with its norm $\|\cdot\|_\infty$. 
The map ${\cal F}:U \to U$ is defined by
$$
{\cal F}(u):= 
u-C(1,\beta(u))Ru-D(1,\beta(u))B(\beta(u),u),
$$
where
\be 
\label{betaudef}
 [\beta(u)](t,x):=
(\partial_{u_1}b_1(x,u(t,x),\partial_{u_2}b_2(x,u(t,x)).
\ee
The element $u^0 \in U$ in Theorem \ref{app1}
is our function $u^0 \in C_{per}(\R \times [0,1];\R^2)$
from assumption \reff{unull}.
Because of \reff{unull} and of Lemma \ref{equiv}
we have ${\cal F}(u^0)=0$.  

The Lie group 
$\Gamma$ is the rotation group 
SO(2)$=S^1=\{e^{2\pi i \vp}: \vp \in \R\}$, 
and the representation is $S_\vp$ defined by \reff{repdef}.
It is easy to verify that
\be 
\label{invar}
\left.
\begin{array}{l}
S_\vp \beta(u)=\beta(S_\vp u),\;
S_\vp B(\beta,u)=B(S_\vp \beta,S_\vp u),\\
S_\vp C(T,\beta)R=C(T,S_\vp \beta)RS_\vp,
S_\vp D(T,\beta)=D(T,S_\vp \beta)S_\vp. 
\end{array}
\right\}
\ee
Therefore, we have $S_\vp\F(u)=\F(S_\vp u)$
for all $u \in U$.

In order to state a formula for 
$\F'(u)$, one has to use the definition of the map $\F$ and the chain rule. But, if $u$ is a solution to the equation $\F(u)=0$, then the formula 
for $\F'(u)$
is surprisingly simple. Indeed, from \reff{artif} it follows that,
\be 
\label{Fder}
\mbox{if $\F(u)=0$, then } \F'(u)=
I-C(1,\beta(u))R -D(1,\beta(u))\partial_u B(\beta(u),u).
\ee
In particular, $\F'(u^0)=
I-C(1,\beta^0)R -D(1,\beta^0)\partial_u B(\beta^0,u^0)$. Here we used that $\beta(u^0)=\beta^0$ (cf. \reff{betanulldef}).
Hence, Lemma \ref{Fredholm} claims that ${\cal F}'(u^0)$ is Fredholm of index zero from $U$ into $U$.
Because the map $\F$ is $C^2$-smooth, Theorem  \ref{app1} yields also that  the map $\vp \in \R \mapsto S_\vp u^0\in U$
is $C^2$-smooth, i.e. $\partial_t^2u^0$ exists and belongs to $C_{per}(\R \times [0,1];\R^2)$.
\begin{remark}
\label{counter}\rm
The following example shows that, if  
\reff{aass}, \reff{aassnot} and \reff{unull} are satisfied but neither \reff{nonressys} nor \reff{nonressys1}, then it may happen that $\de_t^2u^0$ does not exist. We take
$a_1(x)\equiv 1$, $a_2(x)\equiv -1$, $b_j(x,u)\equiv0$, $r_1=r_2=1$
and
$$
u^0_1(t,x)=\Psi(t-x), \; u^0_2(t,x)=\Psi(t+x),
$$
where  $\Psi \in C^1(\R)\setminus C^2(\R)$ is 1-periodic.
\end{remark}

\subsection{Ansatz for the unknown function \boldmath$u$\unboldmath}
\label{Ansatz}
Using Theorem \ref{appl2},
in this section we introduce new coordinates   
close to the orbit
$$
{\cal O}(u^0):=\{S_\vp u^0 \in C_{per}(\R \times [0,1];\R^2)  : \vp \in \R\}
$$
of $u^0$.
We introduce the setting of Theorem \ref{appl2} as follows: The Banach space $U$ is again $C_{per}(\R \times [0,1];\R^2)$ with its norm $\|\cdot\|_\infty$. The Lie group 
$\Gamma$ is the rotation group 
$SO(2)=S^1=\{e^{2\pi i \vp}: \vp \in \R\}$, 
with the representation  $S_\vp$ defined in \reff{repdef}. Hence, the tangential space to 
${\cal O}(u^0)$ at the point $u^0$ is (cf. Lemma \ref{lemker} $(i)$)
\be 
\label{kerT}
T_{u^0}{\cal O}(u^0)=\mbox{span} \{\partial_t u^0\}=\ker (I-C(1,\beta^0)R-D(1,\beta^0)\partial_uB(\beta^0,u^0)).
\ee
We define
\be 
\label{Vdef}
V:=\left\{u \in C_{per}(\R \times [0,1];\R^2):\;
\sum_{j=1}^2\int_0^1\int_0^1u_j(t,x)u_j^*(t,x)\,dtdx=
0\right\}.
\ee
Then $V$ is a closed subspace in $C_{per}(\R \times [0,1];\R^2)$, and assumption \reff{keradsys} yields that
\be 
\label{oplus}
C_{per}(\R \times [0,1];\R^2)
=T_{u^0}{\cal O}(u^0) \oplus V,
\ee
Hence, Theorem \ref{appl2}  implies that
there exists an open neighborhood ${\cal V}_0 \subset V$ of zero such that $\{S_\gamma(u^0+v)\in U:\, \gamma \in \Gamma, \; v \in  
{\cal V}_0\}$ is an open neighborhood of ${\cal O}(u^0)$.

Since Theorems \ref{sys1} and \ref{sys2} are about solutions to \reff{sys} such that $u\approx {\cal O}(u^0)$ in the sense of $\|\cdot\|_\infty$, 
we are allowed to make the following ansatz in \reff{abstract}:
\be 
\label{ansatzeq}
u=S_\vp(u^0+v), \; v \in V.
\ee
Moreover, in what follows we choose the artificial $\beta$ in \reff{abstract} depending on the new state parameter $\vp$ as 
(cf. notion \reff{betanulldef})
\be 
\label{betavpdef}
\beta=S_\vp\beta^0,
\mbox{ i.e. }
\beta(t,x)=\left(\partial_{u_1}b_1(x,u^0(t+\vp,x),
\partial_{u_1}b_1(x,u^0(t+\vp,x)))\right).
\ee
We insert \reff{ansatzeq} and \reff{betavpdef} into \reff{abstract}, apply $S_{\vp}^{-1}$ on the resulting equation, use
\reff{invar} and get 
$$
u^0+v=C(T,\beta^0)(R(u^0+v)+\ve S_\vp^{-1}g)
+D(T,\beta^0)(B(\beta^0,u^0+v)+\ve S_\vp^{-1}f).
$$
Because of $u^0=C(1,\beta^0)Ru^0
+D(1,\beta^0)(B(\beta^0,u^0)$
and the Taylor formulas 
\begin{eqnarray*}
B(\beta^0,u^0+v)&=&B(\beta^0,u^0)+\partial_uB(\beta^0,u^0)v+
\int_0^1(1-s)\partial^2_uB(\beta^0,u^0+sv)(v,v)\,ds,\\
C(T,\beta^0)w&=&C(1,\beta^0)w+
(T-1)\int_0^1\partial_TC(1+s(T-1),\beta^0)w\,ds,\\
D(T,\beta^0)u&=&
D(1,\beta^0)u+
(T-1)\int_0^1\partial_TD(1+s(T-1),\beta^0)u\,ds,
\end{eqnarray*}
this is equivalent to
\begin{eqnarray}
\label{abstract1}
v&=&C(T,\beta^0)Rv+D(T,\beta^0)\left(
\partial_uB(\beta^0,u^0)v+
\int_0^1(1-s)
\partial^2_uB(\beta^0,u^0+sv)(v,v)ds\right)\nonumber\\
&&+(T-1)\int_0^1\left(\partial_TC(1+s(T-1),\beta^0)Ru^0+\partial_TD(1+s(T-1),\beta^0)B(\beta^0,u^0)\right)ds\nonumber\\
&&+\ve\left(C(T,\beta^0)S_\vp^{-1}g
+D(T,\beta^0)S_\vp^{-1}f\right).
\end{eqnarray}
Here we use the notion $S_\vp$ also for the shift operators on the space $C_{per}(\R;\R^2)$, i.e. $[S_\vp^{-1}g](t):=g(t-\vp)$
for $\vp \in \R$ and $g \in C_{per}(\R;\R^2)$.

The notions $\partial_TC(1+s(T-1),\beta^0)Ru^0$ and 
$\partial_TD(1+s(T-1),\beta^0)B(\beta^0,u^0)$
in \reff{abstract1} have to be used with some care: The map $T\mapsto C(T,\beta)$ is not continuous from $(0,\infty)$ into ${\cal L}\left(C_{per}(\R;\R^2);C_{per}(\R\times[0,1];\R^2)\right)
$ (in the sense of the uniform operator norm).
But the map $T\mapsto C(T,\beta)w$ is $C^1$-smooth from $(0,\infty)$ into 
$C_{per}(\R\times[0,1];\R^2)$
if $\beta \in C_{per}(\R\times[0,1];\R^2)$ and $w \in C_{per}(\R;\R^2)$ are both $C^1$-smooth. This follows from the definitions \reff{aldef} and \reff{Cdef}, which lead, for example, for the first component, to
$$
[C_1(T,\beta)w](t,x)=\exp\left(-\int_0^x
\frac{\beta_1(t+\al_1(x,0)/T,y)}{a_1(y)}\,dy\right)w(t+\al_1(x,0)/T).
$$
In particular, $\beta^0$ and $u^0$ are $C^1$-smooth, hence the limit
$$
\lim_{\tau \to 0}\frac{C(T+\tau,\beta^0)Ru^0-
C(T,\beta^0)Ru^0}{\tau}
$$
exists in $C_{per}(\R\times[0,1];\R^2)$, and this is what we denote by $\partial_TC(T,\beta^0)Ru^0$, for the sake of shortness (and similarly for
$\partial_TD(T,\beta^0)B(\beta^0,u^0)$).

We end up with the following: Solutions $u$ to \reff{sys}  with $u\approx {\cal O}(u^0)$ in the sense of $\|\cdot\|_\infty$ and with \reff{ansatzeq}
are solutions to \reff{abstract1} and vice versa.

\begin{remark}
\label{sub}\rm
If the minimal time-period of $u^0$ is $1/n$ with $n \in N$, then any solution $(\vp,v)$ to \reff{abstract1} also $(\vp+1/n,v)$,
$(\vp+2/n,v)$ etc. are solutions to  \reff{abstract1}, but as solutions to 
\reff{sys}  via
\reff{ansatzeq} they are not different.

If the minimal time-period of the forcings $f$ and $g$ is $1/m$ with $m \in N$, $m>1$, then again along with any solution $(\vp,v)$ to \reff{abstract1} also $(\vp+1/m,v)$,
$(\vp+2/m,v)$ etc. are solutions to  \reff{abstract1}, and as solutions to 
\reff{sys}  via
\reff{ansatzeq} these solutions now are  various, in general.
\end{remark}

\subsection{A priori estimate}
\label{Apriori}
In this subsection we prove the following statement.
\begin{lemma}
\label{apri}
Suppose \reff{aass}, \reff{aassnot}, \reff{unull}, \reff{kersys},
\reff{keradsys}, and
assume that  one of the conditions \reff{nonressys}
and \reff{nonressys1} is satisfied.
Then there exist $\ve_0>0$ and $c>0$ such that for all solutions $(\vp,v) \in \R \times V$ to \reff{abstract1}  with
$\ve+|T-1|+\|v\|_\infty \le \ve_0$
we have
$$
|T-1|+\|v\|_\infty \le c\ve.
$$
\end{lemma}
\begin{proof}
Suppose the contrary. Then there exists a sequence $(\ve_k,T_k,\vp_k,v_k) \in (0,\infty) \times
\R^2\times V$, $k \in \N$, of solutions to \reff{abstract1} such that
$$
\ve_k+|T_k-1|+\|v_k\|_\infty+
\frac{\ve_k}{|T_k-1|+\|v_k\|_\infty}
\to 0 \mbox{ for } k \to \infty.
$$
Without loss of generality, we may assume that $\vp_k \in [0,1]$ for all $k$. Hence,  we may assume (by choosing an appropriate subsequence) that there exists $\vp_0$ such that
$$
\vp_k \to \vp_0 \mbox{ for } k \to \infty.
$$
Similarly, without loss of generality we may assume that there exists $\tau_0$ such that
$$
\tau_k:=\frac{T_k-1}{|T_k-1|+\|v_k\|_\infty} \to \tau_0 \mbox{ for } k \to \infty.
$$ 
If we  divide equation \reff{abstract1} by $|T_k-1|+\|v_k\|_\infty$ and set  $w_k:=v_k/(|T_k-1|+\|v_k\|_\infty)$, then we get
\begin{eqnarray*}
&&w_k-C(T_k,\beta^0)Rw_k\\
&&=D(T_k,\beta^0)
\left(\partial_uB(\beta^0,u^0)w_k+
\int_0^1(1-s)
\partial^2_uB(\beta^0,u^0+sv_k)(v_k,w_k)ds\right)\nonumber\\
&&+\tau_k\int_0^1\left(\partial_TC(1+s(T_k-1),\beta^0)Ru^0+\partial_TD(1+s(T_k-1),\beta^0)B(\beta^0,u^0)\right)ds\nonumber\\
&&+\frac{\ve_k}{|T_k-1|+\|v_k\|_\infty}\left(C(T_k,\beta^0)S_{\vp_k}^{-1}g
+D(T_k,\beta^0)S_{\vp_k}^{-1}f\right).
\end{eqnarray*}
Hence, for $k \to \infty$ we get the convergence
$$
(I-C(T_k,\beta^0)R-D(T_k,\beta^0)\partial_uB(\beta^0,u^0))w_k\to
\tau_0\left(\partial_TC(1,\beta^0)Ru^0+\partial_TD(1,\beta^0)B(\beta^0,u^0)\right)
$$
in $C_{per}(\R\times[0,1];\R^2)$.
From Lemma \ref{Invert} it follows that 
$$
\Big(I-(I-C(T_k,\beta^0))^{-1}D(T_k,\beta^0)\partial_uB(\beta^0,u^0)\Bigr)w_k \mbox{ converges in } C_{per}(\R\times[0,1];\R^2).
$$
Multiplying by $I+(I-C(T_k,\beta^0))^{-1}D(T_k,\beta^0)\partial_uB(\beta^0,u^0)$, we get that also
\be 
\label{conve}
\Big(I-((I-C(T_k,\beta^0))^{-1}D(T_k,\beta^0)\partial_uB(\beta^0,u^0))^2\Big)w_k \mbox{ converges in } C_{per}(\R\times[0,1];\R^2).
\ee
Here we used that $\|(I-C(T_k,\beta^0))^{-1}u\|_\infty \le \mbox{const}\|u\|_\infty$ (cf. Lemma \ref{Invert}) and \reff{besch}.

But for any $k$ the operators $((I-C(T_k,\beta^0))^{-1}D(T_k,\beta^0)\partial_uB(\beta^0,u^0))^2$ are compact in $C_{per}(\R\times[0,1];\R^2)$
(cf. the proof of Lemma \ref{Fredholm}).
Therefore, there exists a subsequence $w^1_1,w^1_2,\ldots$ of the sequence $w_1,w_2,\ldots$ such that
$$
\Big((I-C(T_1,\beta^0))^{-1}D(T_1,\beta^0)\partial_uB(\beta^0,u^0)\Big)^2w^1_k \mbox{ converges in } C_{per}(\R\times[0,1];\R^2).
$$
Similarly, there exists a subsequence $w^2_1,w^2_2,\ldots$ of the sequence $w^1_1,w^1_2,\ldots$ such that
$$
\Big((I-C(T_2,\beta^0))^{-1}D(T_2,\beta^0)\partial_uB(\beta^0,u^0)\Big)^2w^2_k \mbox{ converges in } C_{per}(\R\times[0,1];\R^2).
$$
Proceeding this way and taking the "diagonal" subsequence $w^1_1,w^2_2,\ldots$, which is also a subsequence of $w_1,w_2,\ldots$, we get that
$$
\Big((I-C(T_k,\beta^0))^{-1}D(T_k,\beta^0)\partial_uB(\beta^0,u^0)\Big)^2w^k_k \mbox{ converges in } C_{per}(\R\times[0,1];\R^2).
$$
Hence, \reff{conve} yields that without loss of generality we may assume 
that there exists $w_0 \in V$ such that  $\|w_k-w_0\|_\infty \to 0$, and, hence,
$\|C(T_k,\beta^0)Rw_k-C(1,\beta^0)Rw_0\|_\infty \to 0$ and $\|D(T_k,\beta^0)\partial_uB(\beta^0,u^0)w_k-D_k(1,\beta^0)\partial_uB(\beta^0,u^0)w_0\|_\infty \to 0$
for $k \to \infty$ (cf. \reff{conti}).
It follows that
\begin{eqnarray} 
\label{contra1}
&&(I-C(1,\beta^0)R-D(1,\beta^0)\partial_uB(\beta^0,u^0))w_0\nonumber\\
&&=
\tau_0\left(\partial_TC(1,\beta^0)Ru^0+\partial_TD(1,\beta^0)B(\beta^0,u^0)\right)
\end{eqnarray}
but 
\be 
\label{contra2}
|\tau_0|+\|w_0\|_\infty=1.
\ee
If we apply the functional $\phi$ to both sides of \reff{contra1}, then Lemma \ref{lemker} $(ii)$ 
implies that
\be 
\label{norma1}
\tau_0 \phi(\partial_TC(1,\beta^0)Ru^0 +\partial_TD(1,\beta^0)B(\beta^0,u^0))=0.
\ee
In order to calculate
$\phi\left[\partial_TC(1,\beta^0)Ru^0 +\partial_TD(1,\beta^0)B(\beta^0,u^0)\right]$,
we proceed as follows:
The definitions \reff{Bdef} and \reff{Rdef}--\reff{Ddef} yield that
$
[C_1(T,\beta^0)Ru^0 +D_1(T,\beta^0)B(\beta^0,u^0)](t,0)=r_2u^0_2(t,0).
$
Hence, we have $[\de_TC_1(1,\beta^0)Ru^0 +\de_TD_1(1,\beta^0)B(\beta^0,u^0)](t,0)=0$. Similarly one shows that
$[\de_TC_2(1,\beta^0)Ru^0 +\de_TD_2(1,\beta^0)B(\beta^0,u^0)](t,1)=0$.
Therefore,
\begin{eqnarray*}
&&\phi\left[\de_TC_1(1,\beta^0)Ru^0 +\de_TD_1(1,\beta^0)B(\beta^0,u^0)\right]\\
&&=\sum_{j=1}^2\int_0^1\int_0^1\A_j(1)(
\partial_TC(1,\beta^0)Ru^0 +\partial_TD(1,\beta^0)B(\beta^0,u^0)u_j^* \,dtdx.
\end{eqnarray*}
On the other hand, the identity \reff{AC} implies 
$$
\de_T\A(1)C(1,\beta^0)Ru+\A(1)\de_TC(1,\beta^0)Ru=
\de_T\A(1)D(1,\beta^0)u+\A(1)\de_TD(1,\beta^0)u
=0
$$
for all $u \in  C_{per}(\R\times[0,1];\R^2)$ such that $\de_t^2u$ exist and is continuous. 
Therefore,
\begin{eqnarray*}
&&\phi\left[\de_TC_1(1,\beta^0)Ru^0 +\de_TD_1(1,\beta^0)B(\beta^0,u^0)\right]\\
&&=-\sum_{j=1}^2\int_0^1\int_0^1\A'_j(1)\left[
C(1,\beta^0)Ru^0 +D(1,\beta^0)B(\beta^0,u^0\right]u_j^* \,dtdx
\nonumber\\
&&=\sum_{j=1}^2\int_0^1\int_0^1\de_t\left[
C_j(1,\beta^0)Ru^0 +D_j(1,\beta^0)B(\beta^0,u^0)\right]u_j^*\, dtdx.
\end{eqnarray*}
But \reff{artif} implies that
$$
\de_t\left[C(1,\beta^0)Ru^0 +D(1,\beta^0)B(\beta^0,u^0)\right]=
C(1,\beta^0)R\de_tu^0 +D(1,\beta^0)\de_uB(\beta^0,u^0)\de_tu^0=\de_tu^0,
$$
hence
\be 
\label{norma}
\phi\left[\de_TC_1(1,\beta^0)Ru^0 +\de_TD_1(1,\beta^0)B(\beta^0,u^0)\right]
=1.
\ee
Here we used assumption \reff{keradsys}.

From \reff{norma1} and \reff{norma} it follows 
that $\tau_0=0$, and, therefore, \reff{contra1} implies that
$[I-C(1,\beta^0)R-D(1,\beta^0)\partial_uB(\beta^0,u^0)]w_0=0$.
Hence, \reff{kerT} yields $w_0 \in T_{u^0}{\cal O}(u^0)$. But $v \in V$, hence \reff{oplus} yields $v=0$, which contradicts to
\reff{contra2}.
\end{proof}

\subsection{Proof of Theorem \ref{sys1}}
\label{Proofsys1}
Theorems \ref{sys1} and \ref{sys2} concern solutions 
$(\ve,T,u) \in (0,\infty)^2\times
 C^1_{per}(\R\times[0,1];\R^2)$ 
to \reff{sys} with $\ve \approx 0$ and $T\approx 1$ and $u \approx {\cal O}(u^0)$, i.e. solutions 
$(\ve,T,\vp,v) \in  (0,\infty)^2\times [0,1]\times V$ 
to  \reff{abstract1} with $\ve \approx 0$, $T\approx 1$ and $v \approx 0$. Hence, in order to prove Theorems \ref{sys1} and \ref{sys2}, Lemma \ref{apri} allows to make the following ansatz in \reff{abstract1}:
\be 
\label{ansatz2}
T=1+\ve \tau,\; v=\ve w \mbox{ with } \tau \in \R, \; w \in V.
\ee
Inserting \reff{ansatz2} into \reff{abstract1} and dividing by $\ve$, we get the following equation:
\begin{eqnarray}
\label{abstract2}
F(\ve,\tau,\vp,w)&:=&\left(I-C(1+\ve \tau,\beta^0)R-D(1+\ve \tau,\beta^0)
\partial_uB(\beta^0,u^0)\right)w\nonumber\\
&&-\ve D(1+\ve \tau,\beta^0)\int_0^1(1-s)
\partial^2_uB(\beta^0,u^0+s\ve w)(w,w)\,ds\nonumber\\
&&-\tau\int_0^1\left(\partial_TC(1+s\ve \tau,\beta^0)Ru^0+\partial_TD(1+s\ve \tau,\beta^0)B(\beta^0,u^0)\right)ds\nonumber\\
&&-C(1+\ve \tau,\beta^0)S_\vp^{-1}g
-D(1+\ve \tau,\beta^0)S_\vp^{-1}f=0.
\end{eqnarray}
In particular, for $\ve=0$ we have
\begin{eqnarray}
\label{epsnulleq}
&&(I-C(1,\beta^0)R-D(1,\beta^0)\partial_uB(\beta^0,u^0))w\nonumber\\
&&=\tau\left(\partial_TC(1,\beta^0)Ru^0
+\partial_TD(1,\beta^0)
B(\beta^0,u^0)\right)\nonumber\\
&&\,\;\;\;+C(1,\beta^0)S_\vp^{-1}g
+D(1,\beta^0)S_\vp^{-1}f.
\end{eqnarray}
Because of Lemma \ref{lemker} $(ii)$, this equation is equivalent to the two equations
$PF(0,\tau,\vp,w)=0$ and $(I-P)F(0,\tau,\vp,w)=0$, 
and because of \reff{norma},
the first equation is equivalent to the scalar equation
\be 
\label{IPeg}
\tau+\phi(C(1,\beta^0)S_\vp^{-1}g+D(1,\beta^0)S_\vp^{-1}f)=0.
\ee
In order to determine $\phi(C(1,\beta^0)S_\vp^{-1}g+D(1,\beta^0)S_\vp^{-1}f)$,
we calculate as above. We have
\begin{eqnarray*}
&&[C_1(1,\beta^0)S_\vp^{-1}g+D_1(1,\beta^0)S_\vp^{-1}f](t,0)=g_1(t-\vp),\\
&&[C_2(1,\beta^0)S_\vp^{-1}g+D_2(1,\beta^0)S_\vp^{-1}f](t,1)=g_2(t-\vp).
\end{eqnarray*}
Therefore,
\begin{eqnarray*}
&&\phi(C(1,\beta^0)S_\vp^{-1}g+D(1,\beta^0)S_\vp^{-1}f)\\
&&+\int_0^1\left(
a_2(1)g_2(t-\vp)u_2^*(t,1)
-a_1(0)g_1(t-\vp)u_1^*(t,0)\right)dt\\
&&=\sum_{j=1}^2\int_0^1\int_0^1[\A_j(C(1,\beta^0)S_\vp^{-1}g+D(1,\beta^0)S_\vp^{-1}f)](t,x)u_j^*(t,x)\,dtdx\\
&&=\sum_{j=1}^2\int_0^1\int_0^1
f_j(t-\vp,x)u_j^*(t,x)dtdx,
\end{eqnarray*}
i.e. (cf. \reff{Phidef})
\be 
\label{Phieq}
\phi(C(1,\beta^0)S_\vp^{-1}g+D(1,\beta^0)S_\vp^{-1}f)=-\Phi(\vp).
\ee
Hence, equation \reff{IPeg} is just the phase equation $\Phi(\vp)=\tau$.

In order to prove Theorem \ref{sys1}, we  fix 
 $(\tau,\vp)=(\tau_0,\vp_0)$ to be a solution to the phase equation \reff{IPeg}, and insert it into $(I-P)F(0,\tau,\vp,w)=0$. This way we get 
\begin{eqnarray*}
&&(I-C(1,\beta^0)R-D(1,\beta^0)\partial_uB(\beta^0,u^0))w\\
&&=(I-P)\left(\tau_0\left(\partial_TC(1,\beta^0)Ru^0+\partial_TD(1,\beta^0)B(\beta^0,u^0)\right)+C(1,\beta^0)S_{\vp_0}^{-1}g
+D(1,\beta^0)S_{\vp_0}^{-1}f\right).
\end{eqnarray*}
This equation has a unique solution $w=w_0 \in V$. Indeed, the linear operator $I-C(1,\beta^0)-D(1,\beta^0)\partial_uB(\beta^0,u^0)$ is injective on $V$ because of
\reff{kerT} and
\reff{oplus}, and it is surjective from $V$ onto $\im P$, again because of  \reff{kerT}
and \reff{oplus}.

Let us summarize: Equation \reff{abstract2} has a solution $\ve=0$, $\tau=\tau_0$, $\vp=\vp_0$, $w=w_0$, and we are going to solve
\reff{abstract2} for $\ve\approx 0$, $\tau\approx \tau_0$, $\vp\approx\vp_0$ and $w\approx w_0$ by means of Theorem \ref{app3}.
To this end, we put \reff{abstract2} into the setting of Theorem \ref{app3}. The Banach spaces $U_1$ and $U_2$ are $U_1=\R\times V$ and  $U_2=C_{per}(\R\times[0,1];\R^2)$ with  norms $\|(\vp,w)\|_1=|\vp|+\|w\|_\infty$ and $\|\cdot\|_2=\|\cdot\|_\infty$, respectively. The  vector space $\Lambda=\R^2$ is normed by $\|(\ve,\tau)\|_2=|\ve|+|\tau|$. The map
$\F$ from Theorem \ref{app3}  is just the map $F$ , which is defined in  \reff{abstract2}.
The  control parameter
is $(\ve,\tau)$, and  the state parameter is $(\vp,w)$.
The starting solution $(\lambda_0,u_0)$ is
$(0,\tau_0,\ve_0,w_0)$, i.e. we are going to solve the equation $F(\ve,\tau,\vp,w)=0$ with $(\ve,\tau)\approx (0,\tau_0)$ with respect to
$(\vp,w)\approx (\vp_0,w_0)$.

It is easy to verify that conditions \reff{ass1} and \reff{ass2} of Theorem \ref{app3}  are satisfied. In order to verify the remaining conditions \reff{ass3}--\reff{ass5}, we calculate the partial derivatives
\begin{eqnarray}
\label{dw}
\label{wder}
\partial_wF(\ve,\tau,\vp,w)u&=&\left(I-C(1+\ve \tau,\beta^0)R-D(1+\ve \tau,\beta^0)
\partial_uB(\beta^0,u^0)\right)u\nonumber\\
&&-2\ve D(1+\ve \tau,\beta^0)\int_0^1(1-s)
\partial^2_uB(\beta^0,u^0+s\ve w)(u,w)\,ds\nonumber\\
&&-\ve^2D(1+\ve \tau,\beta^0)\int_0^1(1-s)s
\partial^3_uB(\beta^0,u^0+s\ve w)(w,w,u)\,ds,\\
\label{vphder}
\partial_\vp F(\ve,\tau,\vp,w)
&=&C(1+\ve \tau,\beta^0)S_\vp^{-1}g'+
D(1+\ve \tau,\beta^0)S_\vp^{-1}\partial_tf
\end{eqnarray}
and 
\begin{eqnarray*}
	\partial_w\partial_\vp F(\ve,\tau,\vp,w)u
	&=&0,\nonumber\\
	\partial^2_\vp F(\ve,\tau,\vp,w)
	&=&-C(1+\ve \tau,\beta^0)S_\vp^{-1}g''-
	D(1+\ve \tau,\beta^0)
	S_\vp^{-1}\partial^2_tf,\\
	\partial^2_wF(\ve,\tau,\vp,w)(u,v)&=&
-2\ve D(1+\ve \tau,\beta^0)\int_0^1(1-s)
\partial^2_uB(\beta^0,u^0+s\ve w)(u,v)\,ds\nonumber\\
&&-3\ve^2D(1+\ve \tau,\beta^0)\int_0^1(1-s)s
\partial^3_uB(\beta^0,u^0+s\ve w)(w,u,v)\,ds\\
&&-\ve^3D(1+\ve \tau,\beta^0)\int_0^1(1-s)s^2
\partial^4_uB(\beta^0,u^0+s\ve w)(w,w,u,v)\,ds.
\end{eqnarray*}
Here we used that $\partial_\vp[S_\vp^{-1}g](t)=
\partial_\vp g(t-\vp)=-g'(t-\vp)=-[S_\vp^{-1}g'](t)$ and similarly for $\partial^2_\vp[S_\vp^{-1}g]$, $\partial_\vp[S_\vp^{-1}f]$ and 
$\partial^2_\vp[S_\vp^{-1}f]$.
Using these formulas, it is easy to verify
that the assumption \reff{ass5} of Theorem \ref{app3}  is satisfied.

It remains to verify conditions \reff{ass3} and \reff{ass4} of Theorem \ref{app3}. Both these assumptions concern the operators
\begin{eqnarray*}
\label{dw1}
&&\partial_{(\vp,w)}F(\ve,\tau,\vp_0,w_0)(\vp,w)\\
&&=
\vp\partial_\vp F(\ve,\tau,\vp_0,w_0)+
\partial_wF(\ve,\tau,\vp_0,w_0)w\\
&&=\vp\left(C(1+\ve \tau,\beta^0)S_{\vp_0}^{-1}g'+
D(1+\ve \tau,\beta^0)S_{\vp_0}^{-1}\partial_tf\right)\\
&&\;\;\;\;\;\;+\left(I-C(1+\ve \tau,\beta^0)R-D(1+\ve \tau,\beta^0)
\partial_uB(\beta^0,u^0)\right)w\nonumber\\
&&\;\;\;\;\;\;-2\ve D(1+\ve \tau,\beta^0)\int_0^1(1-s)
\partial^2_uB(\beta^0,u^0+s\ve w_0)(w_0,w)\,ds\nonumber\\
&&\;\;\;\;\;\;-\ve^2D(1+\ve \tau,\beta^0)\int_0^1(1-s)s
\partial^3_uB(\beta^0,u^0+s\ve w_0)(w_0,w_0,w)\,ds
\end{eqnarray*}
with $(\ve,\tau)\approx (0,\tau_0)$.
In order to prove, that these conditions are satisfied for all $(\ve,\tau)\approx (0,\tau_0)$, unfortunately,
it is not sufficient to prove that these assumptions are satisfied for $\ve=0$ and $\tau=\tau_0$. The reason for that is that the operator $\partial_{w}F(\ve,\tau,\vp_0,w_0)$ does not depend continuously (with respect to the uniform operator norm) on $\ve$ and $\tau$.

First, let us verify condition \reff{ass4} of Theorem \ref{app3}.
\begin{lemma}
\label{apri1}
Suppose \reff{aass}, \reff{aassnot}, \reff{unull}, \reff{kersys},
\reff{keradsys}, \reff{beta}, and assume
that one of the conditions \reff{nonressys}
and \reff{nonressys1} is satisfied.
Then there exist $\ve_0>0$ and $c>0$ such that for all $(\ve,\tau,\vp,w) \in [0,\infty)\times\R^2 \times V$ with
$
\ve+|\tau-\tau_0| \le \ve_0
$
we have
$$
\|\partial_{(\vp,w)}F(\ve,\tau,\vp_0,w_0)(\vp,w)\|_\infty \ge c(|\vp|+\|w\|_\infty).
$$
\end{lemma}
\begin{proof}
We proceed as in the proof of Lemma \ref{apri}.
Suppose the contrary. Then there exists a sequence $(\ve_k,\tau_k,\vp_k,w_k) \in (0,\infty) \times
\R^2\times V$, $k \in \N$, with
$|\vp_k|+\|w_k\|_\infty=1$ for all $k$,
but 
\be
\label{lim}
\|\vp_k\partial_\vp F(\ve_k,\tau_k,\vp_0,w_0)+
\partial_wF(\ve_k,\tau_k,\vp_0,w_0)w_k\|_\infty \to 0 \mbox{ for } k\to \infty.
\ee
Without loss of generality we may assume that there exists $\vp_*$ such that
$\vp_k \to \vp_*$ and, hence, 
$$
\vp_k\partial_\vp F(\ve_k,\tau_k,\vp_0,w_0)
\to
\vp_*\partial_\vp F(0,\tau_0,\vp_0,w_0)=
\vp_*\left(C(1,\beta^0)S_{\vp_0}^{-1}g'+
D(1,\beta^0)S_{\vp_0}^{-1}\partial_tf\right).
$$
Then \reff{lim} yields that
$\partial_wF(\ve_k,\tau_k,\vp_0,w_0)w_k$ 
converges in $C_{per}(\R\times[0,1];\R^2)$,
and \reff{besch} and \reff{wder} imply that
$$
(I-C(1+\ve_k \tau_k,\beta^0)R-D(1+\ve_k \tau_k,\beta^0)
\partial_u B(\beta^0,u^0))w_k
\mbox{ converges in } C_{per}(\R\times[0,1];\R^2).
$$ 
As in the proof of Lemma \ref{apri}, by using a "diagonal" subsequence, without loss of generality we may assume that there exists $w_* \in V$ such that $\|w_k-w_*\|_\infty \to 0$ for $k \to \infty$. Hence, \reff{wder}, \reff{vphder} and \reff{lim} imply that
\be
\label{cont1}
\vp_*\left(C(1,\beta^0)S_{\vp_0}^{-1}g'+
D(1,\beta^0)S_{\vp_0}^{-1}\partial_tf\right)
+(I-C(1,\beta^0)-D(1,\beta^0)
\partial_u B(\beta^0,u^0))w_*=0,
\ee
but
\be 
\label{cont2}
|\vp_*|+\|w_*\|_\infty=1.
\ee
We apply the functional $\phi$ (cf. \reff{phidef}) to \reff{cont1} and use \reff{Phieq}
in order to get $\vp_* \Phi'(\vp_0)=0$, i.e. $\vp_*=0$, i.e.
$
(I-C(1,\beta^0)-D(1,\beta^0)
\partial_u B(\beta^0,u^0))w_*=0.
$
But the operator $I-C(1,\beta^0)-D(1,\beta^0)
\partial_u B(\beta^0,u^0)$ is injective on $V$, hence $w_*=0$. This way we get  a contradiction to condition \reff{cont2}.
\end{proof}

Finally, let us verify condition \reff{ass3} of
Theorem \ref{app3}, i.e. the condition that for $\ve \approx 0$ and $\tau \approx \tau_0$ the operator
$\partial_{(\vp,w)}F(\ve,\tau,\vp_0,w_0)$
is Fredholm of index zero from $\R \times V$ into $C_{per}(\R \times[0,1];\R^2)$.
Because of Lemma \ref{apri1}, this is equivalent to the condition that for $\ve \approx 0$ and $\tau \approx \tau_0$ the operator
$\partial_{(\vp,w)}F(\ve,\tau,\vp_0,w_0)$
is an isomorphism from $\R \times V$ onto $C_{per}(\R \times[0,1];\R^2)$.
\begin{lemma}
\label{Fred1}
Suppose \reff{aass}--\reff{unull}, \reff{kersys},
\reff{keradsys}, \reff{beta}, and assume  that  one of the conditions \reff{nonressys}
and \reff{nonressys1} is satisfied.
Then there exists $\ve_0>0$ such that, for all $(\ve,\tau) \in [0,\infty)\times\R$ with
$
\ve+|\tau-\tau_0| \le \ve_0,
$
the operator
$\partial_{(\vp,w)}F(\ve,\tau,\vp_0,w_0)$
is an isomorphism from $\R \times V$ onto $C_{per}(\R \times[0,1];\R^2)$.
\end{lemma}
\begin{proof}
We have
$
\partial_{(\vp,w)}F(\ve,\tau,\vp_0,w_0)(\vp,w)=
I_{\ve,\tau}(\vp,w)+\ve J_{\ve,\tau}w
$
with linear operators $I_{\ve,\tau}$ and $J_{\ve,\tau}$ defined by
\begin{eqnarray*}
I_{\ve,\tau}(\vp,w)&:=&
\vp\left(C(1+\ve \tau,\beta^0)S_{\vp_0}^{-1}g'+
D(1+\ve \tau,\beta^0)S_{\vp_0}^{-1}\partial_tf\right)\\
&&+\left(I-C(1+\ve \tau,\beta^0)R-D(1+\ve \tau,\beta^0)
\partial_uB(\beta^0,u^0)\right)w
\end{eqnarray*}
and
\begin{eqnarray*}
J_{\ve,\tau}w&:=&
-2D(1+\ve \tau,\beta^0)\int_0^1(1-s)
\partial^2_uB(\beta^0,u^0+s\ve w_0)(w_0,w)\,ds\\
&&-\ve D(1+\ve \tau,\beta^0)\int_0^1(1-s)s
\partial^3_uB(\beta^0,u^0+s\ve w_0)(w_0,w_0,w)\,ds.
\end{eqnarray*}
Further, for all $\ve \in [0,1]$, $\tau \in (-1,1)$ and $w \in V$ we have  $\|J_{\ve,\tau}w\|_\infty \le \,$const\,$\|w\|_\infty$, where the constant does not depend on $\ve$, $\tau$ and $w$. Hence, Lemma \ref{apri1} yields that
\be 
\label{apri2}
\|I_{\ve,\tau}(\vp,w)\|_\infty \ge \,\mbox{const}\,(|\vp|+\|w\|_\infty)
\ee
for  $\ve \approx 0$, $\tau \approx \tau_0$, $\vp \in \R$ and $w \in V$,  where the constant is positive and does not depend on $\ve$, $\tau$, $\vp$ and $w$. Below we will show that
\be 
\label{Fred2}
\left.
\begin{array}{l}
\mbox{for all $\ve \approx 0$ and $\tau \approx \tau_0$ the operator $I_{\ve,\tau}$
is Fredholm}\\ \mbox{of index zero 
from $\R \times V$ onto $C_{per}(\R \times[0,1];\R^2)$}.
\end{array}
\right\}
\ee
Then \reff{apri2} yields that for $\ve \approx 0$ and $\tau \approx \tau_0$ the operator $I_{\ve,\tau}$
is an isomorphism 
from $\R \times V$ onto $C_{per}(\R \times[0,1];\R^2)$, and $\|I_{\ve,\tau}^{-1}u\|_\infty \le \mbox{const}\,\|u\|_\infty$,
where the constant does not depend on $\ve$,
$\tau$ and $u$. Hence, the operator
$\partial_{(\vp,w)}F(\ve,\tau,\vp_0,w_0)
=I_{\ve,\tau}+\ve J_{\ve,\tau}$
is an isomorphism from $\R \times V$ onto $C_{per}(\R \times[0,1];\R^2)$ if the operator
$I+\ve I_{\ve,\tau}^{-1}J_{\ve,\tau}$
is an isomorphism from $\R \times V$ onto itself, and this is the case for small $\ve$.

It remains to prove \reff{Fred2}. We proceed as in the proof of Lemma \ref{Fredholm}.
Because of $\ve$ and  $\tau$ are fixed, in the notations below we will not mention the dependence on $\ve$ and $\tau$, i.e. 
$
C:=C(1+\ve\tau;\beta^0),\; D:=D(1+\ve\tau,\beta^0),\; B^0:=\partial_u B(\beta^0,u^0).
$
Then
$$
I_{\ve,\tau}(\vp,w)=
\vp(CS_{\vp_0}^{-1}g'+
DS_{\vp_0}^{-1}\partial_tf)
+(I-CR-DB^0)w.
$$

The projection corresponding to the algebraic sum \reff{oplus} is
$
Q\in {\cal L}(C_{per}(\R \times[0,1];\R^2))$ with 
$$
Qu:=\sum_{j=1}^2\int_0^1\int_0^1u_ju_j^*\,dtdx\; \partial_tu^0,\;\ker Q=V,\;
\im Q=\mbox{span}\{\partial_tu^0\}.
$$
The map $u \mapsto
\left(\sum_{j=1}^2\int_0^1\int_0^1u_ju_j^*\,dtdx,
(I-Q)u\right)$ is an isomorphism from $C_{per}(\R \times[0,1];\R^2)$ onto $\R \times V$.
Hence, to prove \reff{Fred2}, it suffices to show that the linear map
$$
u \mapsto \sum_{j=1}^2\int_0^1\int_0^1u_ju_j^*\,dtdx\;
(CS_{\vp_0}^{-1}g'+
DS_{\vp_0}^{-1}\partial_tf)
+(I-CR-DB^0)(I-Q)u
$$
is Fredholm of index zero from  $C_{per}(\R \times[0,1];\R^2)$ into itself.
Because of Lemma \ref{Invert},  
it suffices to prove that the linear map
$$
u \mapsto \sum_{j=1}^2\int_0^1\int_0^1u_ju_j^*\,dtdx\;
(I-CR)^{-1}(CS_{\vp_0}^{-1}g'+
DS_{\vp_0}^{-1}\partial_tf)
+(I-(I-CR)^{-1}DB^0)(I-Q)u
$$
is Fredholm of index zero from  $C_{per}(\R \times[0,1];\R^2)$ into itself.
But the operators $Q$ and $u \mapsto
\sum_{j=1}^2\int_0^1\int_0^1u_ju_j^*\,dtdx\;
(I-CR)^{-1}(CS_{\vp_0}^{-1}g'+
DS_{\vp_0}^{-1}\partial_tf)$
have finite dimensional images.
Finally, due to  Lemma~\ref{Invert},
  the linear map
$
u \mapsto (I-(I-CR)^{-1}DB^0)u
$
is Fredholm of index zero from  $C_{per}(\R \times[0,1];\R^2)$ into itself,
what finishes the proof.
\end{proof}

Now we  can solve
\reff{abstract2} for $\ve\approx 0$, $\tau\approx \tau_0$, $\vp\approx\vp_0$, $w\approx w_0$ by means of Theorem \ref{app3}.
We get the following result:

Suppose \reff{aass}--\reff{unull}, 
\reff{kersys}, \reff{keradsys}, \reff{beta}, and assume  that  one of the conditions \reff{nonressys} and
\reff{nonressys1} is satisfied.
Then the following is true:

There exist $\ve_0 >0$ and $\delta>0$ such that for all   $(\ve,\tau) \in (0,\ve_0)
\times (\tau_0-\ve_0,\tau_0+\ve_0)$ 
equation \reff{abstract2} has
a unique solution $(\vp,w)=(\vp_{\ve,\tau},w_{\ve,\tau})$ 
  with  
$|\vp-\vp_0|+\|w-w_0\|_\infty<\delta$. Moreover,
the map  $(\ve,\tau)  \mapsto (\vp_{\ve,\tau}, w_{\ve,\tau})$ 
is continuous from  $(0,\ve_0)
\times (\tau_0-\ve_0,\tau_0+\ve_0)$
into $\R \times V$.

Let us translate this result into the language of Theorem \ref{sys1} via
$T=1+\ve \tau$ and $u_{\ve,T}=
S_{\vp_{\ve,\tau}}(u^0+\ve w_{\ve,\tau})$.
Then we have $(\ve,\tau) \in (0,\ve_0)
\times (\tau_0-\ve_0,\tau_0+\ve_0)$ if and only if
$(\ve,T) \in K(\ve_0,\tau_0)$. 
Hence, assertions $(i)$ and $(ii)$ of Theorem \ref{sys1} are satisfied if assertion $(ii)$ of Lemma \ref{equiv} works, i.e. if 
$\partial_tu_{\ve,T}$ exists and is continuous, i.e.  if 
$\partial_tw_{\ve,\tau}$ exists and is continuous.

\begin{lemma}
\label{wreg}
Suppose \reff{aass}--\reff{unull}, \reff{kersys},
\reff{keradsys}, and assume  that  one of the conditions \reff{nonressys}
and \reff{nonressys1} is satisfied.
Then, if $\ve_0$ is sufficiently small, $\partial_tw_{\ve,\tau}$ exists and is continuous for all $\ve>0$ and $\tau \in \R$ with $\ve+|\tau-\tau_0|<\ve_0$. 
\end{lemma}
\begin{proof}
We know that $\vp=\vp_{\ve,\tau}-\psi$, $w=S_\psi w_{\ve,\tau}$ is a solution to equation 
$S_\psi F(\ve,\tau,\vp-\psi,S_\psi^{-1}w)=0$ (cf. \reff{abstract2}).
But \reff{invar} and the definition \reff{abstract2} of  $F$  yields
that
\begin{eqnarray*}
&&\F(\ve,\tau,\psi,\vp,w):=S_\psi F(\ve,\tau,\vp-\psi,S_\psi^{-1}w)\\
&&=\left(I-C(1+\ve \tau,S_\psi \beta^0)R-D(1+\ve \tau,S_\psi \beta^0)
\partial_uB(S_\psi \beta^0,S_\psi u^0)\right)w\nonumber\\
&&-\ve D(1+\ve \tau,S_\psi \beta^0)\int_0^1(1-s)
\partial^2_uB(S_\psi \beta^0,S_\psi u^0+s\ve w)(w,w)\,ds\nonumber\\
&&-\tau\int_0^1\left(\partial_TC(1+s\ve \tau,S_\psi \beta^0)Ru^0+\partial_TD(1+s\ve \tau,S_\psi \beta^0)B(S_\psi \beta^0,S_\psi u^0)\right)ds\nonumber\\
&&-C(1+\ve \tau,S_\psi \beta^0)S_{\psi-\vp}g
-D(1+\ve \tau,S_\psi \beta^0)S_{\psi-\vp}f.
\end{eqnarray*}
Now, the maps $\psi \mapsto S_\psi \beta^0$,
$\psi \mapsto S_\psi u^0$ and $\psi \mapsto S_\psi f$  are $C^1$-smooth from $\R$ into $C_{per}(\R \times [0,1];\R^2)$, and similarly for $\psi \mapsto S_\psi g$. Hence, 
the map $(\psi,\vp,w) \mapsto
\F(\ve,\tau,\psi,\vp,w)$ is $C^1$-smooth from $\R^2 \times V$ into $C_{per}(\R \times [0,1];\R^2)$. 

Let us consider the equation
$\F(\ve,\tau,\psi,\vp,w)=0$.
We use the implicit function theorem not to solve it with respect to $(\vp,w)\approx (\vp_{\ve,\tau},w_{\ve,\tau})$ in terms of $\psi\approx 0$ (because the solution is known, it is $\vp=\vp_{\ve,\tau}-\psi$, $w=S_\psi w_{\ve,\tau}$), but to get the $C^1$-smoothness of the data-to-solution map
$\psi \mapsto (\vp_{\ve,\tau}-\psi,S_\psi w_{\ve,\tau})$.
This  $C^1$-smoothness of the data-to-solution map yields that $\partial_tw_{\ve,\tau}$ exists and is an element of the space
$C_{per}(\R \times [0,1];\R^2)$. 

It remains to show that the implicit function theorem is applicable, i.e. that
the linear operator 
$
\partial_{(\vp,w)}\F(\ve,\tau,0,\vp_{\ve,\tau},w_{\ve,\tau})
=\partial_{(\vp,w)}F(\ve,\tau,\vp_{\ve,\tau},w_{\ve,\tau})
$
is an isomorphism from $\R \times V$ onto
$C_{per}(\R \times [0,1]:\R^2)$ for
$\ve \approx 0$ and $\tau \approx \tau_0$.
But the Lemmas \ref{apri1} and \ref{Fred1}
yield that for $\ve \approx 0$ and $\tau \approx \tau_0$ the linear operator $\partial_{(\vp,w)}F(\ve,\tau,\vp_{0},w_{0})$
is an isomorphism from $\R \times V$ onto
$C_{per}(\R \times [0,1]:\R^2)$, and the operator norm of $\partial_{(\vp,w)}F(\ve,\tau,\vp_{0},w_{0})^{-1}$ is bounded uniformly with respect to $\ve$ and $\tau$.

Let us summarize: We know that $\partial_{(\vp,w)}F(\ve,\tau,\vp_{\ve,\tau},w_{\ve,\tau})$
is an isomorphism from $\R \times V$ onto
$C_{per}(\R \times [0,1]:\R^2)$ if 
\begin{eqnarray}
\label{isom}
&&\partial_{(\vp,w)}F(\ve,\tau,\vp_{0},w_{0})^{-1}\partial_{(\vp,w)}F(\ve,\tau,\vp_{\ve,\tau},w_{\ve,\tau})\nonumber\\
&&=
I+\partial_{(\vp,w)}F(\ve,\tau,\vp_{0},w_{0})^{-1}
\left(\partial_{(\vp,w)}F(\ve,\tau,\vp_{\ve,\tau},w_{\ve,\tau})
-\partial_{(\vp,w)}F(\ve,\tau,\vp_{0},w_{0})\right)
\end{eqnarray}
is an isomorphism from $\R \times V$ onto
$C_{per}(\R \times [0,1]:\R^2)$.
But for $\ve \approx 0$ and $\tau \approx \tau_0$ we know that
$|\vp_{\ve,\tau}-\ve_0|\approx 0$ and $\|w_{\ve,\tau}-w_0\|_\infty\approx 0$. Moreover, the operators $\partial_{(\vp,w)}F(\ve,\tau,\vp,w)$ depend continuously (with respect to the operator norm) on $\vp$ and $w$, uniformly with respect to 
$\ve \approx 0$ and $\tau \approx \tau_0$.
Hence,
$$
\sup_{|\vp|+\|w\|_\infty=1}\left\|
\partial_{(\vp,w)}F(\ve,\tau,\vp_{0},w_{0})^{-1}
\left(\partial_{(\vp,w)}F(\ve,\tau,\vp_{\ve,\tau},w_{\ve,\tau})
-\partial_{(\vp,w)}F(\ve,\tau,\vp_{0},w_{0})\right)(\vp,w)\right\|_\infty\\
$$
is smaller than one
for $\ve \approx 0$ and $\tau \approx \tau_0$, i.e. for those $\ve$ and $\tau$ the operator \reff{isom} is an isomorphism from $\R \times V$ onto
$C_{per}(\R \times [0,1]:\R^2)$.
\end{proof}

Assertion $(iii)$ of Theorem \ref{sys1} follows from
$$
\frac{1}{\ve}\inf_{\vp \in \R}\|u_{\ve,T}-S_\vp u^0\|_\infty=
\frac{1}{\ve}\inf_{\vp \in \R}\|
(S_{\vp_{\ve,\tau}}-S_\vp)(u^0+\ve w_{\ve,\tau}\|_\infty\le \|w_{\ve,\tau}\|_\infty.
$$
Assertion $(iv)$ of Theorem \ref{sys1} (with $\vp_\tau$ from there defined by $\vp_{0,\tau}$ from here) follows from \reff{epsnulleq} and \reff{Phieq} and
from $u_{\ve,1+\ve\tau}-S_{\vp_{0,\tau}}u^0=
(S_{\ve_{\ve,\tau}}-S_{\ve_{0,\tau}})u^0+
\ve w_{\ve,\tau}\to 0$ for  $\ve \to 0$.

\subsection{Proof of Theorem \ref{sys2}}
\label{Proofsys2}
Suppose \reff{aass}--\reff{unull}, 
\reff{kersys} and \reff{keradsys}, and assume that at least one of the conditions \reff{nonressys} and
\reff{nonressys1} is satisfied.
Let
$(\ve_k,T_k,u_k)$, $k\in\N$, be a sequence
of solutions to \reff{sys} with \reff{seqlim}.
Because of Lemma \ref{equiv} this sequence is also a sequence of solutions to \reff{abstract}. Further, Lemma \ref{apri} yields that
$T_k=1+\ve_k\tau_k$ and  $u_k=S_{\vp_k}(u^0+\ve_k w_k)$
where $\tau_k \in \R$ and $w_k \in V$ are  bounded sequences. We, 
therefore, have \reff{abstract2}, i.e.
$
F(\ve_k,\tau_k,\vp_k,w_k)=0.
$
In particular, \reff{ortho}, \reff{Phieq} and the definition \reff{abstract2}  of the map $F$ yield
\begin{eqnarray*}
0&=&\phi(F(\ve_k,\tau_k,\vp_k,w_k))\\
&=&-\tau_k+\Phi(\vp_k)+\phi((C(1,\beta^0)
-C(1+\ve_k \tau_k,\beta^0))Rw_k)\\
&&+\phi((D(1,\beta^0)
-D(1+\ve_k \tau_k,\beta^0))
\partial_uB(\beta^0,u^0))w_k)\\
&&-\ve_k\phi\left( D(1+\ve_k \tau_k,\beta^0)\int_0^1(1-s)
\partial^2_uB(\beta^0,u^0+s\ve_k w_k)(w_k,w_k)\,ds\right)\\
&&+\tau_k \phi(
(\partial_TC(1,\beta^0)-
\partial_TC(1+s\ve_k \tau_k,\beta^0))Ru^0)\\
&&+\tau_k \phi((\partial_TD(1,\beta^0)
-\partial_TD(1+s\ve_k \tau_k,\beta^0))B(\beta^0,u^0))\\
&&+ \phi((C(1,\beta^0)-C(1+\ve_k \tau_k,\beta^0))S_{\vp_k}^{-1}g+
(D(1,\beta^0)
-D(1+\ve_k \tau_k,\beta^0))S_{\vp_k}^{-1}f).
\end{eqnarray*}
Hence, \reff{besch} and \reff{conti} imply that
$\Phi(\vp_k)-\tau_k \to 0$ for $k \to \infty$.

\section{Proofs for second-order equations}
\label{proofeq}
\setcounter{equation}{0}
\setcounter{theorem}{0}
In this section we will prove  Theorems \ref{eq1} and \ref{eq2}. We  suppose that all assumptions of these theorems are satisfied.

\subsection{Formal calculations}
\label{formal}
In this subsection we show how one can find, by simple formal calculations, the
formula \reff{Phidefeq}
for the phase equation.
Similar calculations could also be done in Section \ref{proofssys} in order to get the formula \reff{Phidef}
in advance,
 without the long rigorous calculations.

We insert $T=1+\ve \tau$  and $u=S_\vp(u^0+
\ve v)$ into \reff{eq}, then apply $S_{\vp}^{-1}$ to the differential equation and the boundary conditions, and afterwards divide by $\ve$
and  tend $\ve$ to zero.
Using Taylor expansions like  
$1/(1+\ve \tau)=-\ve \tau +O(\ve)$ and
$1/(1+\ve \tau)^2=-2\ve \tau +O(\ve)$ and also the differential equation and the boundary conditions satisfied by $u^0$, i.e.
$$
\partial_t^2u^0-a(x)^2\partial_x^2u^0+b(x,u^0,\partial_tu^0,\partial_xu^0)=0,\;
u^0(t,0)=\partial_xu^0(t,1)+\gamma u(t,1)=0,
$$
one gets
$$
\begin{array}{l}
\partial_t^2v-2\tau \de_t^2u^0-a(x)^2\partial_x^2v
+b_2(t,x)v
+b_3(t,x)
(\de_tv-\tau\de_tu^0)+
b_4(t,x)
\de_xv=f(t-\vp,x),\\
v(t,0)=g_1(t-\vp),\;
\partial_xv(t,1)+\gamma v(t,1)=g_2(t-\vp).
\end{array}
$$
After  multiplying the differential equation by $u^*$, integrating over $t$ and $x$, and using the boundary conditions for $u^*$ and $v$, one ends up with
\begin{eqnarray*}
&&\int_0^1\int_0^1\left(
f(t-\vp,x)u^*(t,x)+\tau(2\de_t^2u^0(t,x)+
b_3(t,x)\de_tu^0(t,x))\right)
 u^*(t,x)\,dtdx\\
&&=\int_0^1\left[-a^2\de_xvu^*+
v((a(2a'+\gamma a)+b_4)u^*
+a^2\de_xu^*) 
\right]_{x=0}^{x=1}
dt\\
&&=-\int_0^1\left(a(1)^2g_2(t-\vp)u^*(t,1)+
a(0)^2g_1(t-\vp)\de_xu^*(t,0)\right)
dt.
\end{eqnarray*}\
Because of \reff{keradeq} and \reff{Phidefeq},
this is just the phase equation $\Phi(\vp)=\tau$.

\subsection{Transformation of the second-order equation into a first-order system}
\label{Transfo}
In this subsection we show that any solution $u$ to  problem \reff{eq} for a second-order equation creates a solution 
\be 
\label{trafo}
\displaystyle v_1(t,x):=\frac{1}{T}\de_tu(t,x)+a(x)\de_xu(t,x),\quad
\displaystyle v_2(t,x):=\frac{1}{T}\de_tu(t,x)-a(x)\de_xu(t,x)
\ee
to the following problem for a first-order  system of integro-differential
equations:
\be
\label{FOS}
\left.
  \begin{array}{l}
 \displaystyle   \frac{1}{T}\de_tv_1(t,x)-a(x)\de_xv_1(t,x)+[\B(\ve,v)](t,x)=\ve f(t,x),\\
 \displaystyle    \frac{1}{T}\de_tv_2(t,x)+a(x)\de_xv_2(t,x)+[\B(\ve,v)](t,x)=\ve f(t,x),\\
 \displaystyle    v_1(t,0)+v_2(t,0)= \frac{2\ve}{T} g_1'(t),\\
  \displaystyle   v_1(t,1)-v_2(t,1)+\gamma a(1)\int_0^1\frac{v_1(t,y)-v_2(t,y)}{a(y)}\,dy
    =2\ve a(1)g_2(t),\\
    v(t+1,x)=v(t,x).
  \end{array}
\right\}
\ee
And vice versa. Here the nonlinear operator $\B$ is defined by
\begin{eqnarray}
\label{RBdef}
 && [\B(\ve,v)](t,x)\nonumber\\
 &&:=b\Bigl(x,\ve g_1(t)+[J v](t,x), [Kv](t,x),[L v](t,x)\Bigr)+\frac{a'(x)}{2}(v_1(t,x)-v_2(t,x))
\end{eqnarray}
with a partial integral operator $J$ defined by
$$
[Jv](t,x):=\frac{1}{2}\int_0^x\frac{v_1(t,y)-v_2(t,y)}{a(y)}\,dy
$$
and with "pointwise" operators $K$ and $L$ defined by 
$$
[Kv](t,x):=\frac{v_1(t,x)+v_2(t,x)}{2}, \; [Lv](t,x):=\frac{v_1(t,x)-v_2(t,x)}{2a(x)}.
$$

As in Section \ref{proofssys}, we denote by $C^1_{per}(\R\times [0,1];\R^2)$ the space of all $C^1$-smooth functions $v:\R \times  [0,1] \to \R^2$ such that
$v(t+1,x)=v(t,x)$ for all $t \in \R$ and $x \in [0,1]$.
Similarly, by $C^2_{per}(\R\times [0,1];\R)$  we denote the space of all $C^2$-smooth functions $u:\R \times  [0,1] \to \R^2$ such that
$u(t+1,x)=u(t,x)$ for all $t \in \R$ and $x \in [0,1]$.

\begin{lemma}\label{lem:FOS}
For all $\ve>0$ and $T>0$  the following is true:

(i) If $u \in C_{per}^2(\R\times [0,1];\R)$ is a solution to \reff{eq}, then the function $v \in C^1_{per}(\R\times [0,1];\R^2)$,
which is defined by \reff{trafo},
is a solution to \reff{FOS}.

(ii) Let $v \in C_{per}^{1}(\R\times [0,1];\R^2)$ be a solution to \reff{FOS}. Then the function  $u$, which is defined by
\be 
\label{udef}
u(t,x)=\ve g_1(t)+\frac{1}{2}\int_0^x\frac{v_1(t,y)-v_2(t,y)}{a(y)}\,dy,
\ee
is $C^2$-smooth and is a solution to \reff{eq}.
Moreover, if $\de^2_tv$ exists and is continuous, then $\de_t^3u$ exists and is continuous also.
\end{lemma}
\begin{proof}
  $(i)$ Let $u \in C_{per}^2(\R\times [0,1];\R^2)$ be given, and let $v \in 
 C^1_{per}(\R\times [0,1];\R^2)$ be defined by \reff{trafo}. Then
\be
  \label{KJ}
  \de_tu=T\frac{v_1+v_2}{2}=Kv,\;  \de_xu=\frac{v_1-v_2}{2a}=L v
  \ee
and $\de_tv_1=\de^2_tu/T+a\de_t\de_xu$,  
$\de_xv_1=\de_t\de_xu/T+a' \de_xu + a\de_x^2u$, $\de_tv_2=\de^2_tu/T-a\de_t\de_xu$ and  $\de_xv_2=\de_t\de_xu/T-a' \de_xu - a\de_x^2u$.
Hence
\be
  \label{uv}
  \frac{1}{T^2}\de_t^2u-a^2\de_x^2u-aa'\de_xu=
  \frac{1}{T}
  \de_tv_1-a\de_xv_1=
  \frac{1}{T}
  \de_tv_2+a\de_xv_2.
  \ee
Further, let $u$ be a solution to problem \reff{eq}.  
Then \reff{KJ} and the boundary conditions $u(t,0)=\ve g_1(t)$ and $\de_xu(t,1)+\gamma u(t,1)=\ve g_2(t)$ imply that $v_1(t,0)+v_2(t,0)=2\ve g'_1(t)/T$ and
  $v_1(t,1)-v_2(t,1)+\gamma a(1) [Lv](t,1)=2\ve a(1)g_2(t)$, i.e. the boundary conditions of \reff{FOS}.
Moreover, from $u(t,0)=\ve g_1(t)$ and \reff{KJ} it follows that $u(t,x)=\ve g_1(t)+[Jv](t,x)$.
  Hence, \reff{KJ}, \reff{uv} and the differential equation in \reff{eq} yield the differential equations in \reff{FOS}.

 $(ii)$ Let $v \in C_{per}^{1}(\R\times [0,1];\R^2)$ be a solution to \reff{FOS}, and 
let $u \in C_{per}^{1}(\R\times [0,1];\R)$ be defined by \reff{udef}.
Then
\begin{eqnarray*}
  \de_tu(t,x)&=&
  \ve g_1'(t)+\frac{1}{2}\int_0^x\frac{\de_tv_1(t,y)-\de_tv_2(t,y)}{a(y)}\,dy\\
  &=&\ve g_1'(t)+\frac{T}{2}
 \int_0^x(\de_xv_1(t,y)+\de_xv_2(t,y)\,dy=
 \frac{T}{2}
 (v_1(t,x)+v_2(t,x)).
  \end{eqnarray*}
  Here we used the first boundary condition and the differential equations in \reff{FOS}.
It follows that $\de_tu$ is $C^{1}$-smooth, and
  \be
  \label{uv1}
  \de^2_tu=\frac{T}{2}(\de_tv_1+\de_tv_2).
  \ee
  Further, \reff{udef} yields that $\de_xu=\de_x Jv=L v$,
  i.e.  $\de_xu$ is $C^{1}$-smooth also, i.e.  $u$ is $C^{2}$-smooth,
  and $2(a'\de_xu+a\de_x^2u)=\de_xv_1-\de_xv_2$, i.e.
   \be
  \label{uv3}
  a^2\de^2_xu=\frac{a}{2}(\de_xv_1-\de_xv_2)-\frac{a'}{2}(v_1-v_2).
  \ee
  But \reff{FOS}, \reff{uv1} and \reff{uv3} imply
the differential equation in \reff{eq}. 

The first boundary condition in \reff{eq} follows from \reff{udef}, and 
the second boundary conditions in \reff{FOS}
follows from $\de_xu=Kv$ and from the second boundary condition in \reff{FOS}.

Finally, suppose that $\de^2_tv$ exists and is continuous.
Then \reff{uv1} yields that also $\de_t^3u$ exists and is continuous.
\end{proof}

Unfortunately, we cannot apply  Theorems \ref{sys1} and \ref{sys2} directly to system \reff{FOS} because of the two small differences between \reff{FOS} and \reff{sys}:
First, there is an $\ve$-dependence in the nonlinearity $B(\ve,v)$ which leads to an additional (in comparison with \reff{Phidef}) term in the formula for the function $\Phi$.
Second, the equations and one boundary condition in \reff{FOS} are nonlocal. Therefore, we have to check again whether  the linearization of \reff{FOS} creates a Fredholm operator of index zero. 

Hence, we adapt the content of Section \ref{proofssys} to the situation of system \reff{FOS}.

\subsection{Fredholmness}
\label{Fredholmn}
As in  Section \ref{proofssys}, we introduce artificial terms $\beta_j(t,x)v_j(t,x)$ with some $\beta \in C_{per}(\R\times [0,1];\R^2)$ 
and rewrite \reff{FOS} as 
\be
\label{FOS1}
\left.
  \begin{array}{l}
\displaystyle    \frac{1}{T}\de_tv_1(t,x)-a(x)\de_xv_1(t,x)+\beta_1(t,x)v_1(t,x)=[B_1(\ve,\beta,v)](t,x)+\ve f(t,x),\\
 \displaystyle    \frac{1}{T}   
    \de_tv_2(t,x)+a(x)\de_xv_2(t,x)+\beta_2(t,x)v_2(t,x)=[B_2(\ve,\beta,v)](t,x)+\ve f(t,x),\\
 \displaystyle    
    v_1(t,0)+v_2(t,0)= \frac{2\ve}{T}  g_1'(t),\\
  \displaystyle      v_1(t,1)-v_2(t,1)+\gamma a(1)\int_0^1\frac{v_1(t,y)-v_2(t,y)}{a(y)}\,dy
    =2\ve a(1)g_2(t),\\
    v(t+1,x)=v(t,x),
  \end{array}
\right\}
\ee
where
\be
\label{BBdefi}
    [B_j(\ve,\beta,v)](t,x):=\beta_j(x,\la)v_j(t,x)
    -[\B(\ve,v)](t,x),\; j=1,2.
  \ee
Similar to Subsection \ref{Char},  system \reff{FOS} of integro-differential equations is equivalent to a system of partial integral equations. In order to state this more precisely,
for $j=1,2$, $t \in \R$, $x,y \in [0,1]$,
$\beta \in C_{per}(\R\times [0,1];\R^2)$
and  $T>0$,  we write
\begin{eqnarray*}
&&c_1(t,x,y,\beta,T):=\exp\int_y^x
\frac{\beta_1(t-\al(x,z)/T,z)}{a(z)}\,dz,\\
&&c_2(t,x,y,\beta,T):=\exp\int_x^y
\frac{\beta_2(t+\al(x,z)/T,z)}{a(z)}\,dz,
\end{eqnarray*}
where the function $\al$ is defined in \reff{bjk1}.
The system of partial integral equations is  (for $(t,x) \in \R \times [0,1]$)
\be
\label{parinteq}
\left.
\begin{array}{l}
\displaystyle v_1(t,x)-c_1(t,x,0,\beta,T)[-v_2(s,0)+2\ve g'_1(s)/T]_{s=t-\al(x,0)/T}\\
=\displaystyle
-\int_0^x\frac{c_1(t,x,y,\beta,T)}{a(y)}[B_1(\ve,\beta,v)+\ve f](t-\al(x,y)/T,y)\,dy,\\
\displaystyle v_2(t,x)-c_2(t,x,1,\beta,T)[v_1(s,1)
+2\ve a(1)g_2(s)]_{s=t+\al(x,1)/T}\\
\displaystyle =
\gamma a(1)c_2(t,x,1,\beta,T)\int_0^1\frac{v_1(t+\al(x,1)/T,y)-v_2(t+\al(x,1)/T,y)}{a(y)}\,dy\\
\displaystyle \;\;\;\;-\int_x^1\frac{c_2(t,x,y,\beta,T)}{a(y)}[B_2(\ve,\beta,v)+\ve f](t+\al(x,y)/T,y)\,dy,
\end{array}
\right\}
\ee
and, as in Subsection \ref{Char}, one can prove the following statement.
\begin{lemma}
\label{equiveq}
For all $\ve>0$,
$\beta \in C_{per}(\R\times [0,1];\R^2)$
and $T>0$ the following is true:

(i) Any solution to \reff{FOS1} is a solution to \reff{parinteq}.

(ii) If $v \in C_{per}(\R\times[0,1];\R^2)$ is a solution to \reff{parinteq} and if the partial derivatives $\partial_tv$ and $\de_t\beta$ exist and are continuous, then $v$ is $C^1$-smooth and $v$ is a solution to \reff{FOS1}.
\end{lemma}

Further, as in Subsection \ref{Char}, we write system \reff{parinteq} in an abstract form. 
For that we introduce  linear operators 
\begin{eqnarray*}
[Rv](t)&:=&(-v_2(t,0),v_1(t,1)),\\
\mbox{$[C(\beta,T)w](t,x)$}&:=&
\left[
\begin{array}{l}
c_1(t,x,0,\beta,T)w_1(t-\al(x,0)/T)\\
c_2(t,x,1,\beta,T)w_2(t+\al(x,1)/T)
\end{array}
\right],\\
\mbox{$[D(\beta,T)v](t,x)$}&:=&
\left[
\begin{array}{c}
\displaystyle-\int_0^x\frac{c_1(t,x,y,\beta,T)}{a(y)}v_1(t-\al(x,y)/T,y)\,dy\\
~\displaystyle-\int_x^1\frac{c_2(t,x,y,\beta,T)}{a(y)}v_2(t+\al(x,y)/T,y)\,dy
\end{array}
\right],\\
\mbox{$[E(\beta,T)v](t,x)$}&:=&
\left[
\begin{array}{c}
0\\
~\displaystyle
\gamma a(1)c_2(t,x,1,\beta,T)\int_0^1\frac{v_1(t+\al(x,1)/T,y)-v_2(t+\al(x,1)/T,y)}{a(y)}\,dy
\end{array}
\right]
\end{eqnarray*}
and functions $\tilde{f}:\R \times 
[0,1]\to\R^2$ and $\tilde{g}:\R \to\R^2$ and by
$$
\tilde{f}(t,x):=
\left[
\begin{array}{c}
f(t,x)\\
f(t,x)
\end{array}
\right],\;
\tilde{g}(t):=
\left[
\begin{array}{c}
2g_1'(t)\\
2a(1)g_2(t)
\end{array}
\right].
$$
Then problem \reff{parint} can be written
analogously to \reff{abstract} as
\be
\label{abstracteq}
v=C(\beta,T)(Rv+\ve \tilde{g}/T) +D(\beta,T)(B(\ve,\beta,v)+\ve \tilde{f})+E(\beta,T)v.
\ee

In what follows we will work with the solution
$v^0 \in  C^1_{per}(\R\times [0,1];\R^2)$ to \reff{FOS} with $\ve=0$ and $T=1$, which corresponds 
to the solution $u^0$ to \reff{eq}
with $\ve=0$ and $T=1$ via \reff{trafo}, i.e.
\be 
\label{vnulldef}
v^0_1(t,x):=\de_tu^0(t,x)+a(x)\de_xu^0(t,x),\;
v^0_2(t,x):=\de_tu^0(t,x)-a(x)\de_xu^0(t,x).
\ee
In order to calculate the linearization with respect to $v$ in $\ve=0$ and $v=v^0$ of the operator $\B$ (cf. \reff{RBdef}), we use the notation \reff{bjk1}.
We have
$$
\partial_v\B(0,v^0)v
=-a'(v_1-v_2)/2+
b_2J v+b_3K v+b_4Lv
=\beta^0_1v_1+\beta^0_2v_2
+b_2Jv
$$
with coefficients $\beta^0_j$ defined in \reff{betadef}.
From \reff{BBdefi} it follows that
  \be
\label{lineB}
\partial_vB_1(0,\beta^0,v^0)v=\beta_2v_2-b_2J v,\quad
\partial_vB_2(0,\beta^0,v^0)v=\beta_1v_1-b_2J v.
\ee
Hence, the linearization with respect to $v$ in $\ve=0$, $\beta=\beta^0$ and $v=v^0$ of the nonlinearity $B$ has a special structure: It is the sum
of a partial integral operator and of a ``pointwise'' operator, which has vanishing diagonal part.

\begin{lemma}
\label{Fredholmeq}
For all $T\approx 1$
the operator $I-C(\beta^0,T)R-D(\beta^0,T)\partial_vB(0,\beta^0,v^0)
+E(\beta^0,T)$
is Fredholm of index zero from $C_{per}(\R\times [0,1];\R^2)$ into itself.
\end{lemma}
\begin{proof}
The conditions \reff{nonressys} and \reff{nonressys1} with $a_1=-a$, $a_2=a$, $r_1=-1$ and $r_2=1$ are just the conditions  
\reff{nonreseq} and \reff{nonreseq1}.
Hence, Lemma \ref{Invert} yields that for $T \approx 1$ the operator $I-C(\beta^0,T)R$ is an isomorphism from $C_{per}(\R\times[0,1];\R^2)$ onto itself, and $(I-C(\beta^0,T)R)^{-1}$ is bounded uniformly with respect to $T \approx 1$.
Therefore, we can proceed as in the proof of Lemma \ref{Fredholm}.

Because of the Fredholmness criterion of Nikolskii, it suffices to show that the operator 
$
\left((I-C(\beta^0,T)R)^{-1}(
D(\beta^0,T)\partial_vB(0,\beta^0,v^0)
+E(\beta^0,T))\right)^2
$
is compact from $C_{per}(\R\times[0,1];\R^2)$ into itself.
Further, because of \reff{lineB}, we have
$$
\de_vB(0,\beta^0,v^0)=B^0+\J \mbox{ with }
B^0v:=(\beta^0_2v_2,\beta^0_1v_1),\;
\J v:=-(b_2Jv,b_2Jv).
$$
In what follows, we will not mention the dependence 
of the operators and the coefficient functions
on $\beta^0$ and $T$.
As in the proof of Lemma \ref{Fredholm} (cf. \reff{prod}), it suffices to show that
the operators 
\begin{eqnarray*}
&&(D(B^0+\J)-E)^2\\
&&=DB^0DB^0+D\J D\J+E^2+DB^0D\J+D\J DB^0-DB^0E-EDB^0-D\J E-ED\J
\end{eqnarray*}
and 
$$
(D(B^0+\J)-E)CR=DB^0CR+D\J CR-ECR
$$
are compact 
from $C_{per}(\R\times [0,1];\R^2)$ into itself.
Since in the proof of Lemma \ref{Fredholm} we already showed  that the operators $DB^0D$ and $DB^0CR$ are compact
from $C_{per}(\R\times [0,1];\R^2)$ to itself,
 it suffices to show that  the operators
\be 
\label{operators}
D\J, \; DB^0E, \; ED, \; E^2
\mbox{ and } ECR
\ee
are compact
from $C_{per}(\R\times [0,1];\R^2)$ into itself.

Let us start with $D\J$. 
The first component of this operator works as follows:
\begin{eqnarray*}
&&[D_1\J v](t,x)\\
&&=-\frac{1}{2}
\int_0^x\frac{c_1(t,x,y)}{a(y)}b_2(t-\al(x,y)/T,y)
\int_0^y
\frac{v_1(t-\al(y,z)/T,z)-v_2(t-\al(y,z)/T,z)}{a(z)}\,dzdy\\
&&=\int_0^x\frac{1}{a(z)}\int_z^x
\frac{c_1(t,x,y)}{a(y)}
[b_2(s,y)(v_1(s,z)-v_2(s,z))]_{s=t-\al(x,y)/T}\,dydz.
\end{eqnarray*}
In the inner integral we change the integration variable $y$ by
$\eta=\widehat{\eta}(t,y,z):=t-\al(y,z)/T$, hence
$d\eta=\frac{dy}{Ta(y)}$.
If $y=\widehat{y}(t,\eta,z)$ is the inverse transformation, then we get
\begin{eqnarray*}
&&[D_1\J v](t,x)\\
&&=\int_0^x\frac{1}{a(z)}\int_t^{t-\al(x,z)/T}
\frac{c_1(t,x,wide{y}(t,\eta,z))}{a(\widehat{y}(t,\eta,z))}
b_1(t,\widehat{y}(t,\eta,z)(v_1(\eta ,z)-v_2(\eta ,z))d\eta \,dz.
\end{eqnarray*}
Now one can see that, for all $v \in C_{per}(\R\times [0,1];\R^2)$,
the function
$D_1\J v$ is $C^1$-smooth, and similarly for the second component $D_2\J v$. Moreover, we have
$$
\sup\{\|\de_tD\J v\|_\infty+
\|\de_xD\J v\|_\infty: \; v \in C_{per}(\R\times [0,1];\R^2), \; 
\|v\|_\infty=1\}<\infty.
$$
Hence the Arzela-Ascoli Theorem yields the needed result.

Now, let us consider $DB^0E$.
Because of $E_1=0$,
we have $B_2^0E=0$ and, hence, $D_2(T)B_2^0E=0$. The first component $D_1B^0E$
works as 
\begin{eqnarray*}
[D_1B^0E v](t,x)
=
\int_0^1
\int_0^xd(t,x,y,z,) [v_1(s,z)-v_2(s,z)]_{s=t-(\al(x,y)-\al(y,1))/T,z)}\,dydz,
\end{eqnarray*}
where
$$
d(t,x,y,z):=-\gamma\frac{a(1)}{a(y)a(z)}
c_1(t,x,y)c_2(t-\al(x,y)/T,y,1)
b_{2}^0(t-\al(x,y)/T,y).
$$
In the inner integral we change the integration variable $y$ to
$\eta=t-(\al(x,y)-\al(y,1))/T$,
 and further proceed as above.

Next we  consider $ED$.
Since $E_1=0$, we are reduced to deal with
 the second component $E_2D$, which 
works as 
\begin{eqnarray*}
[E_2D v](t,x)
&=&
\int_0^x
\int_0^1d_1(t,x,y,z) v_1(t+(\al(x,1)-\al(y,z))/T,z)\,dydz\\
&&+\int_x^1\int_0^1d_2(t,x,y,z) v_2(t+(\al(x,1)+\al(y,z))/T,z)\,dydz,
\end{eqnarray*}
where  $d_1$ and $d_2$ are certain smooth functions, which do not depend on $v$.
In the first inner integral we change the integration variable $y$ to
$\eta=t-(\al(x,1)-\al(y,z))/T$,
and in the second inner integral to
$\eta=t-(\al(x,1)+\al(y,z))/T$.
Afterwards,  we can proceed as above.

Now, let us consider $E^2$.
Because of $E_1=0$
we have to consider the second component $E_2E$, which 
works as 
\begin{eqnarray*}
&&[E_2E v](t,x)\\
&&=
\int_0^1
\int_0^1d(t,x,y,z) [v_1(s,z)-v_2(s,z)]_{s=t+(\al(x,1)+\al(y,1))/T,z)}\,dydz
\end{eqnarray*}
with a smooth function $d$, which does not depend on $v$.
In the inner integral we change the integration variable $y$ to
$\eta=t+(\al(x,1)+\al(y,1))/T$,
and we again can proceed as above.

Finally, let us consider the last operator in \reff{operators}, which is $ECR$. The second component works as 
\begin{eqnarray*}
[E_2CR v](t,x)&=&
\int_0^1
d_1(t,x,y) v_1(t+(\al(x,1)+\al(y,0))/T,1)\,dy\\
&&+\int_0^1
d_2(t,x,y) v_2(t+(\al(x,1)+\al(y,0))/T,0)\,dy
\end{eqnarray*}
with smooth functions $d_1$ and $d_2$, which do not depend on $v$.
In both integrals we change the integration variable $y$ by
$\eta=t+(\al(x,1)-\al(y,0))/T$,
and we can proceed as above.
\end{proof}

Now, as in Subsection \ref{Regularity}, one can show that the solution $v^0$ to
 the autonomous nonlinear problem \reff{FOS} with $\ve=0$ and $T=1$ has additional regularity, i.e. additionally to the $C^1$-smoothness of $v^0$ the second derivative $\de^2_tv^0$ exists and is continuous.
Hence, Lemma \ref{lem:FOS} $(ii)$ yields, that also
the solution $u^0$ 
to the autonomous nonlinear problem \reff{eq} with $\ve=0$ and $T=1$ has additional regularity, namely that $\de_t^3u^0$ exists and is continuous.

\begin{remark}
\label{countereq}\rm
The following example shows that, if  
\reff{unull1} is satisfied but neither \reff{nonreseq} nor \reff{nonreseq1} is satisfied, then it may happen that $\de_t^3u^0$ does not exist.
We take
$a(x)\equiv 4$,  $b(x,u,v,w)\equiv0$, $\gamma=0$
and
$$
u^0(t,x)=\frac{1}{8}\int_0^x
\Psi(4t+y)+\Psi(4t-y))\,dy,
$$
where  $\Psi \in C^1(\R)\setminus C^2(\R)$ satisfies $\Psi(t+2)=-\Psi(t)$ for all $t \in \R$.
Then  
$$
\de_tu^0(t,x)=\frac{1}{2}\int_0^x
\Psi'(4t+y)+\Psi'(4t-y))\,dy=\frac{1}{2}(\Psi(4t+x)-\Psi(4t-x)).
$$
Hence, $\de^2_tu^0(t,x)=2(\Psi'(4t+x)-\Psi'(4t-x))$ is not differentiable with respect to $t$.
On the other hand,
$$
\de_xu^0(t,x)=\frac{1}{8}
\Psi(4t+x)+\Psi(4t-x)),
\mbox{ i.e. }
\de^2_xu^0(t,x)=\frac{1}{8}
\Psi'(4t+x)-\Psi'(4t-x)),
$$
i.e. $\de^2_tu^0(t,x)=16\de_x^2u^0(t,x)$ and $u^0(t,0)=\de_xu^0(t,1)=0$.
\end{remark}

\subsection{Proofs of Theorems \ref{eq1} and \ref{eq2}}
\label{Proofeq1}
Analogously to Lemma \ref{lemker} one can prove that
\begin{eqnarray*}
&&\ker(I-C(\beta^0,1)R-D(\beta^0,1)\partial_uB(0,\beta^0,v^0))=\mbox{span}\{\de_tv^0\},\\
&&\im(I-C(\beta^0,1)R-D(\beta^0,1)\partial_uB(\beta^0,u^0))=\ker \phi,\\
&&C_{per}(\R\times[0,1];\R^2)
=\ker \phi\oplus
\mbox{span}\{D(\beta^0,1)\de_tv^0\},
\end{eqnarray*}
where the functional $\phi$ is defined analogously to \reff{phidef} by
\begin{eqnarray} 
\label{phidefeq}
\phi(v)&:=&\int_0^1\int_0^1
\left([\A_1v](t,x)v_1^*(t,x)+[\A_2v](t,x)v_2^*(t,x)\right)\,dtdx\nonumber\\
&&-\int_0^1\left(a(0)v_1(t,0)v_1^*(t,0)
+a(1)v_2(t,1)v_2^*(t,1)\right)dt
\end{eqnarray}
for $v \in C^1_{per}(\R\times[0,1];\R^2)$.
The operator $\A$ from $C^1_{per}(\R\times[0,1];\R^2)$ into $C_{per}(\R\times[0,1];\R^2)$ is defined analogously to \reff{Adef} by
\begin{eqnarray*}
&&[\A_1v](t,x):=\partial_tv_1(t,x)-a(x)\partial_xv_1(t,x)+
\beta_1(t,x)v_1(t,x),\\
&&[\A_2v](t,x):=\partial_tv_2(t,x)+a(x)\partial_xv_2(t,x)+
\beta_1(t,x)v_2(t,x).
\end{eqnarray*}
Moreover, $v^* \in C^1_{per}(\R\times[0,1];\R^2)$ 
is the solution to the linear homogeneous problem 
\be 
\label{U}
\left.
\begin{array}{l}
-\de_tv_1(t,x)+\de_x(a(x)v_1(t,x))
+\beta_1(t,x)(v_1(t,x)+v_2(t,x))\\
\displaystyle
\;\;\;\;\;\;=-\frac{1}{2a(x)}\left(\int_x^1b_2(t,y)(v_1(t,y)+v_2(t,y))\,dy
+\gamma a(1)^2v_2(t,1)\right),\\
-\de_tv_2(t,x)-\de_x(a(x)v_2(t,x))
+\beta_2(t,x)(v_1(t,x)+v_2(t,x))\\
\displaystyle
\;\;\;\;\;\;=\frac{1}{2a(x)}\left(\int_x^1b_2(t,y)(v_1(t,y)+v_2(t,y))\,dy
+\gamma a(1)^2v_2(t,1)\right),\\
\displaystyle v_1(t,0)+v_2(t,0)= 
v_1(t,1)-v_2(t,1)=0,\;
v(t+1,x)=v(t,x)
\end{array}
\right\}
\ee
with normalization
\be 
\label{norm2} 
\sum_{j=1}^2\int_0^1\int_0^1\de_tv_j^0(t,x)v_j(t,x)\,dtdx=1.
\ee
The following calculation shows that \reff{U} is the adjoint problem to the linearization of \reff{FOS} with respect to $v$ in $\ve=0$, $T=1$ and $v=v^0$:
For all $v,w \in C^1_{per}(\R\times [0,1];\R^2)$ with $v_1(t,0)+v_2(t,0)=0$ and $v_1(t,1)-v_2(t,1)
+\gamma a(1)\int_0^1(v_1(t,y)-v_2(t,y))/a(y)\,dy=0$
we have
\begin{eqnarray*}
&&\int_0^1\int_0^1\Big[\de_tv_1(t,x)-a(x)\de_x v_1(t,x)+\beta_1(t,x)v_1(t,x)+\beta_2(t,x)v_2(t,x)\\
&&+\frac{b_2(t,x)}{2}
\int_0^x\frac{v_1(t,y)-v_2(t,y)}{a(y)}\,dy\Big]w_1(t,x)\,dtdx\\
&&+\int_0^1\int_0^1\Big[\de_t v_2(t,x)+a(x)\de_x v_2(t,x)+\beta_1(t,x)v_1(t,x)+\beta_2(t,x)v_2(t,x)\\
&&+\frac{b_2(t,x)}{2}
\int_0^x\frac{v_1(t,y)-v_2(t,y)}{a(y)}\,dy\Big]w_2(t,x)\,dtdx\\
&&=\int_0^1\int_0^1\bigg[-\de_t w_1(t,x)+\de_x(a(x)w_1(t,x))+\beta_1(t,x)(w_1(t,x)+w_2(t,x))\\
&&+\frac{1}{2a(x)}\left(\gamma a(1)^2w_2(1)+\int_x^1b_2(t,y)(w_1(t,y)+w_2(t,y))\,dy\right)\bigg]v_1(t,x)\,dtdx\\
&&+\int_0^1\int_0^1\bigg[-\de_t w_2(t,x)-\de_x(a(x)w_2(t,x))+\beta_2(t,x)(w_1(t,x)+w_2(t,x))\\
&&+\frac{1}{2a(x)}\left(\gamma a(1)^2w_2(1)+\int_x^1b_2(t,y)(w_1(t,y)+w_2(t,y))\,dy\right)\bigg]v_2(t,x)\,dtdx\\
&&+\int_0^1\big[a(0)(w_1(t,0)+w_2(t,0))v_1(t,0)
+a(1)(-w_1(t,1)+w_2(t,1))v_1(t,1)\big]dt.
\end{eqnarray*}

\begin{lemma}
\label{vstern}
We have $v_1^*=(u^*-\tilde{u})/2$ and
$v_2^*=(u^*+\tilde{u})/2$ with
\be 
\label{tildeu}
\tilde{u}(t,x):=\frac{1}{a(x)}\int_x^1\left(
\de_tu^*(t,y)-b_3(t,y)u^*(t,y)\right)\,dy.
\ee
\end{lemma}
\begin{proof}
 Let us show that the functions $v_1:=(u^*-\tilde{u})/2$ and $v_2:=(u^*+\tilde{u})/2$
satisfy \reff{U} and \reff{norm2}.

We have $v_1(t,0)+v_2(t,0)=u^*(t,0)=0$ because of the boundary condition of $u^*$ at $x=0$ (cf. \reff{adeq}). Moreover, we have $v_1(t,1)-v_2(t,1)=\tilde{u}(t,1)=0$ because of \reff{tildeu}.

Now, to verify that the integro-differential equations in \reff{U}
are satisfied, we have to show that
\begin{eqnarray*}
&&-\de_t(u^*-\tilde{u})+
\de_x(a(u^*-\tilde{u}))
+2\beta_1u^*\\
&&+\frac{1}{a(x)}\left(\int_x^1b_2(t,y)u^*(t,y)\,dy
+\frac{\gamma a(1)^2}{4}u^*(t,1)\right)
=0,\\
&&-\de_t(u^*+\tilde{u})-
\de_x(a(u^*+\tilde{u}))
+2\beta_2u^*\\
&&-\frac{1}{a(x)}\left(\int_x^1b_2(t,y)u^*(t,y)\,dy
+\frac{\gamma a(1)^2}{4}u^*(t,1)\right)
=0.
\end{eqnarray*}
This is satisfied if and only if the sum and the difference of these two equations are satisfied, i.e. 
\begin{eqnarray}
\label{1}
&&-\de_tu^*-
\de_x(a\tilde{u})
+(\beta_1+\beta_2)u^*=0,\\
\label{2}
&&\de_t\tilde{u}+
\de_x(au^*)
+(\beta_1-\beta_2)u^*\nonumber\\
&&+\frac{1}{a(x)}\left(\int_x^1b_2(t,y)u^*(t,y)\,dy
+\gamma a(1)^2u^*(t,1)\right)
=0
\end{eqnarray}
Equation \reff{1} is satisfied because of 
$\beta_1+\beta_2=b_3$ (cf. \reff{betadef})
and  \reff{tildeu}. 

In order to verify equation \reff{2}, we  use the definition \reff{tildeu} of $\tilde{u}$ 
and the fact that $u^*$ satisfies \reff{adeq}. It follows that
\begin{eqnarray}
\label{dtu}
\de_t\tilde{u}(t,x)&=&\frac{1}{a(x)}\int_x^1\left(
\de^2_tu^*(t,y)-\de_t(b_3(t,y)u^*(t,y))\right)\,dy\nonumber\\
&=&\frac{1}{a(x)}\int_x^1\left(
\de^2_y(a(y)^2u^*(t,y))-b_2(t,y)u^*(t,y)+\de_x(b_4(t,y)u^*(t,y))\right)\,dy\nonumber\\
&=&-\frac{1}{a(x)}\int_x^1
b_2u^*\,dy+\frac{1}{a(x)}
\left[(\de_x(a(x)^2u^*(t,x))+b_4(t,y)u^*(t,y)\right]_{y=x}^{y=1}\nonumber\\
&=&-\frac{1}{a(x)}\int_x^1
b_2u^*\,dy-\left(2a'(x)+\frac{b_4(t,x)}{a(x)}\right)u^*(t,x)-a(x)\de_xu^*(t,x)\nonumber\\
&&-\frac{\gamma a(1)^2}{a(x)}u^*(t,1)\nonumber\\
&=&-\frac{1}{a(x)}\int_x^1
b_2u^*\,dy-\left(a'(x)+\frac{b_4(t,x)}{a(x)}\right)u^*(t,x)-\de_x(a(x)u^*(t,x))\nonumber\\
&&-\frac{\gamma a(1)^2}{a(x)}u^*(t,1).
\end{eqnarray}
Since $\beta_1-\beta_2=a'+b_4/a$ (cf. \reff{betadef}) this is just equation \reff{2}.

Finally, let us verify the normalization condition \reff{norm2}. Using  normalization condition \reff{keradeq},  definition \reff{vnulldef} of $v^0$, equation \reff{1} and the boundary conditions  $u^0(t,0)=\tilde{u}(t,1)=0$, we get
\begin{eqnarray*}
&&\sum_{j=1}^2\int_0^1
\int_0^1\de_tv_j^0v_j\,dtdx\\
&&=\frac{1}{2}\int_0^1
\int_0^1\left((\de^2_tu^0+a\de_t\de_xu^0)(u^*-\tilde{u})+
(\de^2_tu^0-a\de_t\de_xu^0)(u^*+\tilde{u})
\right)\,dtdx\\
&&=\int_0^1
\int_0^1\left(\de^2_tu^0u^*
+a\de_t\de_xu^0\tilde{u}\right)\,dtdx
=\int_0^1
\int_0^1\left(\de^2_tu^0u^*
-\de_tu^0\de_x(a\tilde{u})\right)\,dtdx\\
&&=\int_0^1
\int_0^1\left(\de^2_tu^0u^*
-\de_tu^0(\de_tu^*-b_3u^*)\right)\,dtdx
=\int_0^1
\int_0^1\left(2\de^2_tu^0u^*
+b_3\de_tu^0u^*\right)\,dtdx=1.
\end{eqnarray*}
\end{proof}

Now we proceed as in Subsection \ref{Ansatz}.
We make the following
ansatz in \reff{abstracteq}:
$$
T=1+\ve \tau,\; v=S_\vp(v^0+\ve w) \mbox{ with } \tau \in \R \mbox{ and } \sum_{j=1}^2\int_0^1\int_0^1w_jv_j^*\,dtdx=0,\;\beta=S_\vp\beta^0.
$$
Inserting this into \reff{abstracteq}, applying $S_\vp^{-1}$, dividing by $\ve$ and tending $\ve$ to zero,  we get
\begin{eqnarray}
\label{epsnulleq1}
&&(I-C(\beta^0,1)R-D(\beta^0,1)(\partial_uB(0,\beta^0,v^0))w-
D(\beta^0,1)\de_\ve B(0,\beta^0,v^0))
\nonumber\\
&&=\tau\left(\partial_TC(\beta^0,1)Rv^0
+\partial_TD(\beta^0,1)
B(0,\beta^0,v^0)\right)\nonumber\\
&&\;\;\;\;+C(\beta^0,1)S_\vp^{-1}\tilde{g}
+D(\beta^0,1)S_\vp^{-1}\tilde{f}.
\end{eqnarray}
This is an analogue of \reff{epsnulleq}, 
with one additional term $D(\beta^0,1)\de_\ve B(0,\beta^0,v^0))$.
As in Section \ref{proofssys}, we have
$$
\phi((I-C(\beta_0,1)R-D(\beta_0,1)(\partial_uB(0,\beta_0,v^0))w)=0
$$ 
and 
$\phi(\partial_TC(\beta_0,1)Rv^0
+\partial_TD(\beta_0,1)
B(0,\beta_0,v^0))=1$ (cf. \reff{norma}).
Hence, 
$$
\tau=-\phi(C(\beta_0,1)S_\vp^{-1}\tilde{g}+
D(\beta_0,1)(\de_\ve B(0,\beta_0,v^0)
+S_\vp^{-1}\tilde{f})).
$$
In order to prove Theorem \ref{eq1}, it remains to show that this is the phase equation, i.e. that
\begin{eqnarray}
\label{a}
&&\phi(C(\beta_0,1)S_\vp^{-1}\tilde{g}+
D(\beta_0,1)(\de_\ve B(0,\beta_0,v^0)
+S_\vp^{-1}\tilde{f}))\nonumber\\
&&=
\int_0^1\int_0^1f(t-\vp,x)u^*(t,x)\,dtdx\nonumber\\
&&\;\;\;\;+\int_0^1\left(a(1)^2g_2(t-\vp)u^*(t,1)+
a(0)^2g_1(t-\vp)\de_xu^*(t,0))\right)\,dt.
\end{eqnarray}

To prove \reff{a}, let $h_j:=C_j(\beta_0,1)S_\vp^{-1}\tilde{g}+
D_j(\beta_0,1)(\de_\ve B(0,\beta_0,v^0)
+S_\vp^{-1}\tilde{f})$ for $j=1,2$.
Then \reff{phidefeq} yields
\begin{eqnarray}
\label{b}
&&\phi(C(\beta_0,1)S_\vp^{-1}\tilde{g}+
D(\beta_0,1)(\de_\ve B(0,\beta_0,v^0)
+S_\vp^{-1}\tilde{f}))=\sum_{j=1}^2(\de_\ve B_j(0,\beta_0,v^0)
+S_\vp^{-1}\tilde{f}_j)v_j^*ß,dtdx\nonumber\\
&&-\int_0^1\left(a(0)h_1(t,0)v_1^*(t,0)+a(1)h_2(t,1)v_2^*(t,1)\right)\,dt.
\end{eqnarray}
 Since $\tilde{f}_j=f$ for $j=1,2$, then, due to  \reff{RBdef} and \reff{BBdefi}, 
$\de_\ve B_j(0,\beta_0,v^0)](t,x)=-b_2(t,x)g_1(t-\vp)$. Hence, \reff{tildeu} implies the equality
$$
\sum_{j=1}^2(\de_\ve B_j(0,\beta_0,v^0)
+S_\vp^{-1}\tilde{f}_j)v_j^*\,dtdx
=\int_0^1\int_0^1(-b_2(t,x)g_1(t-\vp)+f(t-\vp,x))
u^*(t,x)\,dtdx.
$$
Moreover, $h_1(t,0)=\tilde{g}_1(t-\vp)=2g_1'(t-\vp)$, $h_2(t,1)=\tilde{g}_2(t-\vp)=2a(1)g_2(t-\vp)$, $v_1^*(t,0)=\tilde{u}(t,0)/2$ and 
$v_2^*(t,1)=u^*(t,1)/2$. Therefore,
\begin{eqnarray}
\label{c}
&&\int_0^1\left(a(0)h_1(t,0)v_1^*(t,0)+a(1)h_2(t,1)v_2^*(t,1)\right)dt\nonumber\\
&&=\int_0^1\left(a(0)g_1'(t-\vp)\tilde{u}(t,0)+a(1)^2g_2(t-\vp)u^*(t,1)\right)dt\nonumber\\
&&=\int_0^1\left(-a(0)g_1(t-\vp)\de_t\tilde{u}(t,0)+a(1)^2g_2(t-\vp)u^*(t,1)\right)dt.
\end{eqnarray}
On the other hand, \reff{dtu} yields
$$
\de_t\tilde{u}(t,0)=
\frac{1}{a(0)}
\int_0^1b_2(t,y)u^*(t,y)\,dy+a(0)\de_xu^*(t,0).
$$
This together with \reff{b} and \reff{c} implies \reff{a}.

Finally, Theorem \ref{eq2} follows directly from Theorem \ref{sys2}.

\section{Appendix 1}
\label{appendix1}
\setcounter{equation}{0}
\setcounter{theorem}{0}
Let $U$ be  a Banach space
and  $\Gamma$  a compact Lie group.
Moreover, let $\gamma \in \Gamma \mapsto S_\gamma \in {\cal L}(U)$
 be a strongly continuous representation  of $\Gamma$ in $U$, i.e.
\be 
\label{rep}
\left.
\begin{array}{l}
S_\gamma \circ S_\delta=S_{\gamma \delta},\;
S_e=I,\\
\gamma \in \Gamma \mapsto S_\gamma u \in U
\mbox{ is continuous for all } u \in U.
\end{array}
\right\}
\ee
The following theorem is due to E. N. Dancer (see \cite[Theorem 1 and Remark 4]{Dancer}). Roughly speaking, it claims the following: The map $\gamma \in \Gamma \mapsto S_\gamma u \in U$ is not $C^1$-smooth, in general, but it is if $u$ solves an equivariant equation $\F(u)=0$ with a $C^1$-Fredholm map $\F:U \to U$.
\begin{theorem}
\label{app1}
Suppose \reff{rep} and  let a $C^1$-map $\F:U \to U$ be given such that 
$$
S_\gamma \F(u)=\F(S_\gamma u) \mbox{ for all }
\gamma \in \Gamma \mbox{ and } u \in U.
$$ 
Moreover, let $u^0 \in U$ be given  such that 
$$
\F(u^0)=0,
\mbox{ and $\F'(u^0)$  is Fredholm of index zero from $U$ into $U$.}
$$
Then the following is true:

(i) The map $\gamma \in \Gamma \mapsto S_\gamma u^0 \in U$ is $C^1$-smooth.

(ii) If $\F$ is $C^2$-smooth, then the map $\gamma \in \Gamma \mapsto S_\gamma u^0 \in U$ is $C^2$-smooth also.
\end{theorem}
The following theorem is due to E. N. Dancer and A. Vanderbauwhede
(see \cite{V80,V82,V83}). It describes a parametrization of a tubular neighborhood of the orbit of a "good" element $u^0 \in U$:
\begin{theorem}
\label{appl2}
Suppose \reff{rep}, and let  $u^0 \in U$ be given such that the map $\gamma \in \Gamma \mapsto S_\gamma u^0 \in U$ is $C^1$-smooth.
Then the $\Gamma$-orbit
${\cal O}(u^0):=\{S_\gamma u^0 \in U: \; \gamma \in \Gamma\}$ of $u^0$
is a $C^1$-submanifold in $U$, and for any closed subspace $V$ in $U$ with
$$
U=T_{u^0}{\cal O}(u^0) \oplus V
$$
there exists an open neighborhood ${\cal V}_0 \subset V$ of zero such that $\{S_\gamma(u^0+v)\in U:\, \gamma \in \Gamma, \; v \in  
{\cal V}_0\}$ is an open neighborhood of ${\cal O}(u^0)$.
\end{theorem}

\section{Appendix 2}
\label{appendix2}
\setcounter{equation}{0}
\setcounter{theorem}{0}
Let  $(U_1,\|\cdot\|_1)$ and $(U_2,\|\cdot\|_2)$
be  Banach spaces, $(\Lambda,\|\cdot\|_\Lambda)$ be a normed vector space  and $\F:\Lambda\times U_1 \to U_2$ be a map. The following theorem is folklore, and we prove it for the convenience of the reader. It is a version of the implicit function theorem for cases, such that $\partial_u\F$ exists, but $\partial_u\F(\cdot,u)$ may be discontinuous, and $\F(\cdot,u)$ may be non-differentiable.
\begin{theorem}
\label{app3} 
Let $\lambda_0 \in \Lambda$ and $u_0 \in U_1$
with $\F(\lambda_0,u_0)=0$ be given  elements. Suppose that
\begin{eqnarray}
\label{ass1}
&&\F(\cdot,u) \mbox{ is continuous for all } u \in U_1,\\
\label{ass2}
&&\F(\lambda,\cdot) \mbox{ is $C^2$-smooth for all } \lambda \in \Lambda.
\end{eqnarray}
Moreover, suppose that there exist $\ve_0>0$ and $c>0$ such that for all $\lambda \in \Lambda$ and all $u,v,w \in U_1$ with $\|\lambda-\lambda_0\|_\Lambda+\|u-u_0\|_1<\ve_0$ it holds that
\begin{eqnarray}
\label{ass3}
&&\partial_u\F(\lambda,u_0) \mbox{ is Fredholm of index zero from $U_1$ into $U_2$},\\
\label{ass4}
&&\|v\|_1 \le c\|\partial_u\F(\lambda,u_0)v\|_2,\\\label{ass5}
&&\|\partial^2_u\F(\lambda,u)(v,w)\|_2
\le c\|v\|_1\|w\|_1.
\end{eqnarray}
Then the following is true:

(i) There exist $\ve\in (0,\ve_0)$ and $\delta>0$ such that for all $\lambda \in \Lambda$ with $\|\lambda-\lambda_0\|_\Lambda<\ve$ there exists a unique $u=\hat{u}(\lambda) \in U_1$  with $\|u-u_0\|_1<\delta$ and $\F(\lambda,u)=0$.

(ii) The data-to-solution map $\hat{u}$ is continuous.
\end{theorem} 
\begin{proof}
 $(i)$ Take $\lambda \in \Lambda$
with $\|\la-\la_0\|_\lambda<\ve_0$. Due to \reff{ass3} and \reff{ass4}, the operator
$\partial_u\F(\lambda,u_0)$ is an isomorphism  from $U_1$ onto $U_2$. Hence, the equation $F(\lambda,u)=0$ is equivalent to the fixed point problem
$$
u=\G(\lambda,u):=u-\partial_u\F(\lambda,u_0)^{-1}\F(\lambda,u).
$$
For $\delta>0$, write $B_\delta:=\{u \in U_1:\; \|u-u_0\|_1 \le \delta\}$. Take $u \in B_\delta$ and $v \in U$. Then
$$
\left(I-\partial_u\F(\lambda,u_0)^{-1}\partial_u\F(\lambda,u)\right)v=
\partial_u\F(\lambda,u_0)^{-1}\int_0^1
\partial^2_u\F(\lambda,u_0+s(u-u_0))(u-u_0,v)ds.
$$
Hence, if $\delta<\ve_0$, then \reff{ass4} and \reff{ass5} yield that
$\|\left(I-\partial_u\F(\lambda,u_0)^{-1}\partial_u\F(\lambda,u)\right)v\|_1
\le c^2\delta \|v\|_1$.
Therefore, for $u_1,u_2 \in B_\delta$ we get
\begin{eqnarray*}
&&\|\G(\la,u_1)-\G(\la,u_2)\|_1\nonumber\\
&&=\left\|\int_0^1\left(I-\partial_u\F(\lambda,u_0)^{-1}\partial_u\F(\lambda,su_1+(1-2)u_2)\right)(u_1-u_2)ds\right\|_1\nonumber\\
&&\le\frac{1}{2}\|u_1-u_2\|_1 \mbox{ for } \delta:=\frac{1}{2}\min\left\{\frac{1}{c^2},\ve_0\right\}.
\end{eqnarray*}
Moreover, for $u\in B_\delta$ we have
$$
\|\G(\la,u)\|_1 \le \|\G(\la,u)-\G(\la,u_0)\|_1+
\|\partial_u\F(\lambda,u_0)^{-1}\F(\lambda,u_0)\|_1
\le \frac{\delta}{2}+c\|\F(\la,u_0)\|_2.
$$
Hence, \reff{ass1} yields that for the map $\G(\la,\cdot)$ is strictly contractive from $B_\delta$ into itself. Therefore, the Banach fixed point theorem yields claim $(i)$.

$(ii)$ Take $\la_1,\la_2 \approx \la_0$. Then
\begin{eqnarray*}
&&\|\hat{u}(\la_1)-\hat{u}(\la_2)\|_1=
\|\G(\la_1,\hat{u}(\la_1))-\G(\la_2,\hat{u}(\la_2))\|_1\\
&&\le \|\G(\la_1,\hat{u}(\la_1))-\G(\la_1,\hat{u}(\la_2))\|_1+
\|\G(\la_1,\hat{u}(\la_2))-\G(\la_2,\hat{u}(\la_2))\|_1\\
&&\le \frac{1}{2}\|\hat{u}(\la_1)-\hat{u}(\la_2)\|_1+\|\partial_u\F(\la_1,u_0)^{-1}\F(\la_1,\hat{u}(\la_2))\|_1.
\end{eqnarray*}
Hence, \reff{ass1} yields
$
\|\hat{u}(\la_1)-\hat{u}(\la_2)\|_1
\le 2c \|\F(\la_1,\hat{u}(\la_2))\|_2
\to 0 \mbox{ for } \la_1 \to \la_2.
$
\end{proof}

\begin{remark}
\label{notcont}\rm
Since $\partial_u\F(\cdot,u_0)$ is not supposed to be continuous (with respect to the uniform operator norm in ${\cal L}(U_1;U_2)$), conditions \reff{ass3} and \reff{ass4} are not satisfied for all $\la \approx \lambda_0$, in general, if these conditions are satisfied for $\la=\lambda_0$ only.
\end{remark}

\section*{Acknowledgments}
Irina Kmit was supported by the VolkswagenStiftung Project ``From Modeling and Analysis to Approximation''.

\end{document}